\documentclass[a4paper,11pt]{amsart}

\usepackage[english]{babel}
\usepackage{url}
\usepackage{amsxtra}
\usepackage{mathrsfs}
\usepackage{amssymb}

\numberwithin{equation}{section}

\swapnumbers
\newtheorem{lemma}{Lemma}[section]
\newtheorem{proposition}[lemma]{Proposition}
\newtheorem{theorem}[lemma]{Theorem}
\newtheorem{corollary}[lemma]{Corollary}
\theoremstyle{definition}
\newtheorem{definition}[lemma]{Definition}
\newtheorem{remark}[lemma]{Remark}
\newtheorem{example}[lemma]{Example}

\newtheorem{discussion}[lemma]{Discussion}
\newtheorem{examples}[lemma]{Examples}

\newcommand{\R}{\mathbb{R}}
\newcommand{\C}{\mathbb{C}}
\newcommand{\N}{\mathbb{N}}
\newcommand{\Z}{\mathbb{Z}}
\newcommand{\E}{\mathbb{E}}
\newcommand{\calS}{\mathscr{S}}
\newcommand{\Skr}{{\bf S}^\otimes}
\newcommand{\Mkr}{{\bf M}^\otimes}
\newcommand{\Kkr}{{\bf K}^\otimes}
\newcommand{\Pkr}{{\bf P}^\otimes}
\newcommand{\Tkr}{{\bf T}^\otimes}
\newcommand{\Akr}{{\bf A}^\otimes}
\newcommand{\Ukr}{{\bf U}^\otimes}
\newcommand{\Jkr}{{\bf J}^\otimes}
\newcommand{\Bkr}{{\bf B}^\otimes}
\newcommand{\Psikr}{{\bf \Psi(A)}^\otimes}
\newcommand{\Fkr}{{\bf \cal F}^\otimes}
\newcommand{\Rad}{\operatorname{Rad}}

\let\cal=\mathcal

\begin{document}

\title[The $H^{\infty}$--Functional Calculus and Square Function
Estimates]{The $H^{\infty}$--Functional Calculus 
and Square Function Estimates}

\author[Kalton and Weis]{Nigel J.~Kalton} 
\address{Department of Mathematics\\
University of Missouri\\
Columbia, MO 65201\\
U.S.A.}
\author{Lutz Weis}
\address{Mathematisches Institut I\\
Universit\"at Karlsruhe\\
Englersta{\ss}e 2\\
76128 Karlsruhe\\
Gemany}
\email{Lutz.Weis@math.uni-karlsruhe.de}

\begin{abstract}
Using notions from the geometry of Banach spaces we introduce square
functions $\gamma(\Omega,X)$ for functions with values in an arbitrary Banach space $X$. We show that they have very convenient function space properties comparable to the Bochner norm of $L_2(\Omega,H)$ for a Hilbert space $H$. In particular all bounded operators $T$ on $H$ can be extended to $\gamma(\Omega,X)$ for all Banach spaces $X$. Our main applications are
characterizations of the $H^{\infty}$--calculus that extend known results
for $L_p$--spaces from \cite{CDMY}.
With these square function estimates
we show, e.~g., that a $C_0$--group of operators $T_s$ on a Banach space with
finite cotype has an $H^{\infty}$--calculus on a strip if and only if $e^{-a|s|}T_s$
is $R$--bounded for some $a > 0$. Similarly, a sectorial operator $A$ has an
$H^{\infty}$--calculus on a sector if and only if $A$ has $R$--bounded
imaginary powers. We also consider vector valued Paley--Littlewood 
$g$--functions on $UMD$--spaces.
\end{abstract}

\maketitle


\section{Introduction}

In recent years, the $H^{\infty}$--holomorphic functional calculus for a
sectorial operator on a Banach space has played an important role in the
spectral theory of differential operators and its application to evolution
equations. For example, it is an important tool: in the theory of maximal
regularity for parabolic evolution equations 
(see~\cite{BK1,DDHPV,KKW,KW1,LLL,LeM}) and, very recently, 
the solution of Kato's problem (\cite{AHLMT}).
By now it is known that many systems of elliptic partial differential
operators and Schr\"odinger operators do have an $H^{\infty}$--calculus
(\cite{BK2,DHP,KuWe}). 

It is natural to construct a holomorphic functional calculus for a sectorial
operator $A$ and analytic functions $f$ bounded on a sector $\Sigma$
containing the spectrum of $A$ via the Dunford formula 
$$
f(A) = \frac{1}{2 \pi i} \displaystyle \int\limits_{\partial \Sigma} f(\lambda)
R(\lambda, A)d\lambda.
$$
However, since the resolvent of a sectorial operator grows like $|\lambda|^{-1}$
on $\partial \Sigma$ this integral will be a singular integral in general.
Therefore M. Cowling, J. Doust, A.~McIntosh and A. Yagi used ideas from harmonic analysis to 
characterize the $H^{\infty}$--calculus for sectorial operators $A$
on $L_p(\Omega)$--spaces by square function estimates of the form
\begin{equation}
\biggl| \biggl| \biggl( \int\limits_0^{\infty} |tA(t+A)^{-2} x|^2 \frac{dt}{t}
\biggr)^{1/2} \biggr| \biggr|_{L_p(\Omega)} \simeq \|x\|_{L_p(\Omega)}, 
\end{equation}
or, if $A$ is injective and generates an analytic semigroup $T_t$ and $n \in \N$,
\begin{equation}
\biggl| \biggl| \biggl( \int\limits_0^{\infty} |t^nA^nT_t x
|^2 \frac{dt}{t} \biggr)^{1/2} \biggr| \biggr|_{L_p(\Omega)} \simeq \|x\|_{L_p(\Omega)}
\end{equation}
(for the Hilbert space case see \cite{M}, for $L_p$--spaces \cite{CDMY}).
If $A = (- \Delta)^{1/2}$ on $L_p(\R^n)$, then (2) reduces to the classical
Paley--Littlewood $g$--functions estimates. 
It is therefore not surprising that square function estimates 
proved to be very useful in the theory of the 
$H^{\infty}$--calculus and its applications, most recently in the solution
of Kato's problem (\cite{AHLMT}).

In {\sl Section 4} and {\sl 5} of this paper we introduce a notion of generalized
square functions that will allow us to formulate expressions such as in (1) and (2)
on general Banach spaces. Since in (1) and (2) the function space structure
of $L_p$ is exploited this cannot be done in a straightforward manner;
thus, we use Gaussian random series and the $\gamma$--norm of Banach space theory
in place of the lattice structure. We show that these generalized square
functions have the same formal properties as their classical
counterparts with respect to duality,
integral transforms such as the Fourier--transform or the Hilbert--transform,
multiplication and convolution operators.
An important role plays here is the notion of $R$--boundedness (or rather
$\gamma$--boundedness), which is also a central notion in some recent work on
operator--valued multiplier theorems, maximal regularity and the
$H^{\infty}$--calculus \cite{CPSW,G1,G2,KW1,LeM1,W1}. There are other versions of square
functions on Banach spaces in the literature (see e.~g.~\cite{FM,HM,KW1}),
which lack some of these properties and therefore do not seem to be suitable
to derive the results of this paper.

In many cases square function estimates can be used to extend
Hilbert space results to the Banach space setting.
As an illustration we extend in {\sl Section 5} a 
characterization of group generators on Hilbert space due to Boyadzhiev and
DeLaubenfels (\cite{BD}) to Banach spaces $X$ with finite cotype:
A closed operator $A$ has an $H^{\infty}$--calculus on a vertical strip
around the imaginary axes 
if and only if $A$ generates a $C_0$--group of operators $T_t$ so that
$\{e^{- a|t|}T_t: t \in \R\}$ is $R$--bounded for some $a > 0$ 
(Theorem 6.8).
Furthermore, the group $T_t$ itself is $R$--bounded if and only if 
its generator $A$ is a spectral operator in the sense of Dunford and
Schwartz (Corollary 6.9). The underlying square function estimates also allow 
for the construction of an operator--valued functional calculus
and joint functional calculus.
As a consequence, we obtain that the set of operators $f(A)$ generated by a
uniformly bounded set of analytic functions on a strip is $R$--bounded if $X$
has Pisier's property $(\alpha)$ (see Corollary 6.6).

In {\sl Section 7} we present characterizations of the $H^{\infty}$--calculus
for sectorial operators in terms of square functions. They give some insight
into the gap between the $H^{\infty}$--calculus and the existence of bounded
imaginary powers (BIP): 
Again for a Banach space of finite cotype we show that BIP 
implies the $H^{\infty}$--calculus for $A$, if in addition the set of imaginary
powers $A^{it}$, $t \in [-1,1]$ is $R$--bounded.

In \cite{CDMY} it was asked whether for a sectorial operator with an
$H^{\infty}$--calculus one always has $\omega_{H^{\infty}}(A) = \omega(A)$,
i.~e.~whether the best angle for the $H^{\infty}$--calculus is determined
by the angle of sectoriality (see below for definitions). This is true in
Hilbert space \cite{CDMY}, but not in a general Banach space.
As a partial positive result we point out that $\omega_{H^{\infty}}(H)$ equals 
the angle of almost $R$--sectoriality. Almost $R$--sectorial operators are
introduced here since they provide a more natural framework for the study
of the $H^{\infty}$--calculus than the stronger notion of $R$--sectoriality.

The reader will discover a certain analogy between the results in Section 6
and 7. This is explained in {\sl Section 8}, where we use the logarithm of
a sectorial operator to relate the spectral properties of sectorial operators
and group generators. We compare e.~g.~$R$--boundedness conditions for the
resolvents of $A$ and $\log A$ 
and relate the square functions for $A$ and $\log A$. 

In {\sl Section 9} we come back to the classical Littlewood Paley 
$g$--functions (2) with $A=\Delta$, $A= (- \Delta)^{1/2}$ or, more generally generators of diffusion semigroups. Using our $H^{\infty}$--results
we extend them to Bochner spaces $L_p(\R^N,X)$, $1 < p < \infty$, and 
show that (2) holds if and only if $X$ has the UMD property. One can view
this result as a continuous version of the well known vector--valued 
Paley--Littlewood theorem of J.~Bourgain (\cite{Bo}). These results were obtained independently by T. Hyt\"{o}nen (\cite{Hyt1}) by a different method. They extend work of Xu \cite{Xu},
who assumed that $X$ is a Banach lattice and used a different approach.

Before we consider these Banach space results we recall in {\sl Section 2}
the known Hilbert space results and give, in some cases, simplified proofs.
We do this for two reasons: We need proofs reduced to the 
essentials and as free as possible of unnecessary Hilbert space luxury
to be able to extend them to the Banach space case. Also these proofs show
the workings and essential properties of square functions that will motivate
our general definition in Section 4. {\sl Section 3} serves the same purpose:
we show how square functions in $L_p$ lead to a general approach in the
Banach space setting. 

These results are part of the larger project \cite{KW2}, which will contain
alternative definitions of square functions in terms of Euclidean structures
and, among others, relate results under weaker assumptions on the Banach
spaces.

A first version of this article was circulating among experts since 2002. The second named author would like to apologize for the long delay until providing a final version. Because of this unfortunate circumstance, we would like to mention some papers which built in the meantime on the results of our article in the three areas: Spectral theory and its applications to evolution equations \cite{BHM,FW,HaHa,HHK,HaKu,HaLM,Haa,Haa2,KKW,Kr,Kr2,KU,KuWe,LeM1,LeM2,LeM3,LeM4,Ne2,VeWeis,W}, harmonic analysis of Banach space valued functions \cite{BCFR,H2,Hy,Hyt1,HNP,HyWe2,HW,KaWe,Kr2} and stochastic evolution equations \cite{AHN,CN,CoNe,DNW,HaNV,KNVW,Kun,KuNe1,MN,MN2,NV,NVW1,NVW2,NVW3,NVW4,NVW5,NVW6,NVW7,NW,NW2,PrV,SV,Ver,Ver3,Veraar,
VeWe,VeZ}.

\medskip

{\bf Notations.}
We recall now some basic notations and definitions. For a closed operator
$A$ on a Banach space $X$ we denote by $A'$ the dual operator on the 
dual space $X'$. $X^{0}$ denotes the closure of $\cal{D}(A')$ in $X'$,
which is known to norm $X$. If $X$ is a Hilbert space we denote by 
$(\cdot|\cdot)$ the scalar product on $H$ and by $A^{\ast}$ the Hilbert
space adjoint of an operator $A$. 

In this paper a {\sl sectorial operator}
$A$ is closed, injective, has dense domain ${\cal D}(A)$ and range 
${\cal R}(A)$. Furthermore, for some $\sigma \in (0, \pi)$ we have that
$\sigma (A) \subset \Sigma(\sigma) \cup \{0\}$, where $\Sigma(\sigma) = 
\{ \lambda \in \C: | \arg \lambda | < \sigma \}$ and $\| \lambda R (\lambda, 
A) \| \leq C$ for $\lambda \not\in \Sigma(\sigma)$. $\omega(A)$ is the infimum over all such
$\sigma$. To define an $H^{\infty}$--calculus we first consider 
$H_0^{\infty}(\Sigma(\sigma))$ which contains all bounded analytic functions
$f$ on $\Sigma(\sigma)$ such that $| \lambda|^{- \varepsilon}|f(\lambda)|$
is bounded near $0$ and $| \lambda |^{\varepsilon} |f(\lambda)|$ is bounded
for large $|\lambda|$ for some $\varepsilon > 0$. For $f \in H_0^{\infty}
(\Sigma(\sigma))$ and a sectorial operator $A$ with $\omega(A) < \sigma$
the {\sl Dunford integral}
$$
f(A) = \frac{1}{2 \pi i} \int\limits_{\partial \Sigma(\gamma)} f(\lambda)
R(\lambda, A) d \lambda
$$
exists (where $\omega(A) < \gamma < \sigma)$ and is linear and multiplicative. 
We orientate the curve $\partial
\Sigma(\gamma)$ always in such a way, that the curve surrounds the interior
of $\Sigma(\sigma)$ in a counter clockwise fashion.
This functional calculus can also be extended to
$$
H_1^{\infty}(\Sigma(\sigma)) = \bigl \{f \in H^{\infty}(\Sigma(\sigma)) :
\sup_{|\varphi| < \sigma} \int\limits_0^{\infty}|f(e^{i \varphi}t)|
\frac{dt}{t} < \infty \bigr\}.
$$
We say that $A$ admits an {\sl $H^{\infty}(\Sigma(\sigma))$--functional} calculus if there is a constant
$C$, such that 
$$
\| f(A) \| \leq C \|f\|_{H^{\infty}(\Sigma(\sigma))} \quad \mbox{ for all }
f \in H_0^{\infty}(\Sigma(\sigma)).
$$
By $\omega_{H^{\infty}}(A)$ we denote the infimum over such $\sigma$.
The following convergence lemma from \cite{CDMY} is used often:
If $A$ has an $H^{\infty}$--calculus and $f_n,f \in H^{\infty}(\Sigma(\sigma))$
are uniformly bounded and $f_n(\lambda) \to f(\lambda)$ for $\lambda \in
\Sigma(\sigma)$ then $f_n(A)x \to f(A)x$ for all $x \in X$. 

For a sectorial operator we can define $A^{i s}, s \in \R$, as closed operators
in the sense of fractional powers, see e.~g.~\cite{Ko}.
We say that $A$ has {\sl bounded 
imaginary powers} (BIP) if these operators $A^{is}$ extend to bounded 
operators on $X$. $\omega(A^{is})$ is the growth bound of the group $A^{is}$.

In connection with group generators we also need operators of {\sl strip--type}.
They are again closed injective with dense domain and range but their
spectrum is contained in a strip $S(a) = \{ \lambda : | \mbox{Re} \lambda | 
< a \}$ and $R(\lambda, A)$ is bounded outside $S(a)$. Denote by $w(A)$
the infimum over such $a$. 

Let $H_0^{\infty}(S(a))$ be the space of all bounded analytic functions on
$S(a)$, so that $| \lambda|^{1+ \varepsilon} |f(\lambda)|$ is bounded for
large $| \lambda |$ for some $\varepsilon > 0$. Then, for an operator $A$
of strip--type, we can define an $H^{\infty}(S(a))$--calculus for $a > w(A)$
in the same way we defined an $H^{\infty}$--calculus on $\Sigma(\sigma)$
above (see e.~g.~\cite{Ha1,Ha2} for details). Again $w_{H^{\infty}}(A)$ is the 
infimum over such $a$. We will also need the space
$$
H_1^{\infty}(S(a)) = \bigl \{ f \in H^{\infty}(S(a)) : 
\sup_{|b| < a} \int\limits_{- \infty}^{\infty}|f(b+it)|dt < \infty
\bigr \}.
$$

A set $\tau \subset B(X,Y)$ of operators is called {\sl $R$--bounded}
if there is a constant $C$ so that for all $T_1, \dots , T_m \in \tau$ and
$x_1, \dots , x_m \in X$ 
\begin{equation}
\biggl( \E \biggl| \biggl| \sum_{n=1}^m r_n T_n x_n \biggr| \biggr|_Y^2 
\biggr)^{1/2} \leq C 
\biggl( \E \biggl( \biggl| \biggl| \sum_{n=1}^m r_n x_n \biggr| \biggr|_X 
\biggr)^2 \biggr)^{1/2}
\end{equation}
where $(r_n)$ is the sequence of Rademacher functions. If we replace the
Rademacher functions in inequality (3) by a sequence $(g_n)$ of independent,
$N(0,1)$--distributed Gaussian variables, then the resulting property of
$\tau$ is called {\sl $\gamma$--boundedness}. The smallest constant $C$ 
for which (3) holds is called the $R$--bound ($\gamma$--bound) of $\tau$.
If $T_t$ is a group we consider the growth bound
$$
\omega_R (T_t) = \inf \{ w: e^{-w|t|}T_t \mbox{ is } R \mbox{--bounded} \}
$$
and similarly we define $\omega_\gamma(T_t)$. 
If $X$ has finite cotype, then the notions of $R$--boundedness and
$\gamma$--boundedness are equivalent since in this case there is a constant
$C$ such that for $x_1 , \dots , x_n \in X$ (cf.~\cite{DJT} 12.11, 12.27)
\begin{equation*}
\frac{1}{C} \biggl( \E \biggl| \biggl| \sum_k g_k x_k \biggr| \biggr|^2 
\biggr)^{1/2} \!\! \leq \!\!
\biggl( \E \biggl| \biggl| \sum r_k x_k \biggr| \biggr|^2 \biggr)^{1/2} 
\!\! \leq \! C 
\biggl( \E \biggl| \biggl| \sum g_k x_k \biggr| \biggr|^2 \biggr)^{1/2}.
\end{equation*}
A Banach space has {\sl type} $p$, $p \in [1,2]$, if there is a constant
$C$ such that for all $x_1, \dots , x_m \in X, m \in \N$, we have
$$
\biggl( \E \biggl| \biggl| \sum_{n=1}^m r_n x_n \biggr| \biggr|^2 \biggr)^{1/2} 
\leq C \biggl( \sum_{n=1}^m \| x_n \|^p \biggr)^{1/p}.
$$
The {\sl cotype} $q$, $q \in [2, \infty)$, of a Banach space is defined
similarly by the reverse inequality
$$
\biggl( \sum_{n=1}^m \|x_n \|^q \biggr)^{1/q} \leq C \biggl( \displaystyle \E 
\biggl| \biggl| \sum_{n=1}^m r_n x_n \biggr| \biggr|^2 \biggr)^{1/2}.
$$
Finally we say that $X$ has property $(\alpha)$ (see \cite{Pi2}), if for
two independent sequences $(r_n)$, $(r_n')$ of Rademacher sequences
there is a constant $C$ so that for all $x_{ij} \in X$ and $| a_{ij} | \leq 1$
$$
\biggl( \E \E \biggl| \biggl| \sum_{ij} r_i r_j' a_{ij} x_{ij} 
\biggr| \biggr|^2 \biggr)^{1/2} \leq
C \biggl(\E \E \biggl| \biggl| \sum_{ij} r_i r_j' x_{ij} 
\biggr| \biggr|^2 \biggr)^{1/2}.
$$
A Banach lattice with finite cotype has always property $(\alpha)$.
A uniformly convex space has always finite cotype and a type larger than $1$ 
(cf.~\cite{DJT}). Hence a space $L_r(\Omega)$ with $1 < r < \infty$ has cotype
$q=\min\lbrace r,r' \rbrace <\infty$, type $p=\max\lbrace r,r'\rbrace > 1$ and property $(\alpha)$.

A Banach space $X$ is called a UMD--space, if all $X$--valued martingale difference sequences converge unconditionally. An equivalent property which is more relevant for the present paper is that the Hilbert transform $H$ on $L_2(\R)$ has a bounded tensor extension $H \otimes I$ to $L_2(\R,X)$. For unexplained Banach space notation we refer to \cite{Pi5} and \cite{DJT}.

\section{The Hilbert space case}

In this section we recall the characterization of $H^{\infty}$--calculus
of a sectorial operator $A$ on a Hilbert space $H$ in terms of square
functions such as
$$
\|x\|_A = \biggl ( \int\limits_0^{\infty} \|A^{1/2} R(-t, A)x\|^2 dt
\biggr)^{1/2} = \|A^{1/2}R(\cdot,A)x\|_{L_2(\R_{-},H)}.
$$
The point is to motivate our study of more general square functions 
on Banach spaces in the following sections and to construct proofs that are
simple enough so that they generalize to the Banach space setting.
For that purpose we would like to emphasize that we use only the following
elementary properties of $L_2(I,H)$: 

\begin{itemize}
\item[{$(S1)$}] 
If $f,g \in L_2(I,H)$ then \quad $|(f|g)| \leq \|f\|_{L_2(I,H)} 
\cdot \|g\|_{L_2(I,H)}.$
\item[{$(S2)$}] 
If $t \in I \to N(t) \in B(H)$ is strongly measurable and bounded
(with respect to the operator norm) then
$$
\|N(\cdot)f(\cdot)\|_{L_2(I,H)} \leq \sup_{t \in I} \|N(t)\| \cdot
\|f(\cdot)\|_{L_2(I,H)}.
$$
\item[{$(S3)$}] 
For $S \in B(L_2(I))$ put $\calS
\biggl( \displaystyle \sum_{j=1}^n f_j \otimes x_j \biggr) = \sum_{j=1}^n S(f_j) \otimes x_j$
for $f_j \in L_2(I)$ and $x_i \in H$.
Then $\calS$ extends to a bounded operator from $L_2(I,H)$ to $L_2(I,H)$
and $\| \calS \| \leq \|S\|$.
\end{itemize}
(For a proof of (S3) see e.~g.~\cite{KuWe} Lemma 11.11.)
In particular, for 
$$
Sf(t) = {\cal F}f(t) = \textstyle \frac{1}{\sqrt{2 \pi}}
\int e^{-ist} f(s)ds
$$ 
we obtain the vector--valued Plancherel--identity:
$$
\| {\cal F} f\|_{L_2(\R,H)} = \| f\|_{L_2(\R,H)}.
$$
The following theorem is due to A.~McIntosh \cite{CDMY,M}, but we give a 
different proof:

\begin{theorem}
For a sectorial operator $A$ on a Hilbert space $H$ the following are
equivalent:

\noindent a) 
$A$ has bounded imaginary powers (BIP) and for one (all) $\sigma \in 
(\omega(A),\pi]$ there is a constant $C_1$ with
$$
\|A^{is}\| \leq C_1 e^{\sigma |s|}.
$$

\noindent b)
For one (all) $\omega$ with $| \omega | \in (\omega(A), \pi]$ 
there is a constant $C_2$ such that for all $x \in H$
$$
\biggl( \int\limits_0^{\infty} \| A^{1/2} R(t e^{i \omega}, A)x \|^2 dt
\biggr)^{1/2} \leq C \|x\|, 
$$
and
$$
\biggl( \int\limits_0^{\infty} \| (A^{\ast})^{1/2} R(t e^{i \omega},A)^{\ast}
x\|^2 dt \biggr)^{1/2} \leq C \|x\|.
$$

\noindent c)
For one (all) $\omega$ with $| \omega | \in (\omega(A), \pi]$ there is a $C_3$
such that for all $x \in H$
$$
\frac{1}{C_3} \|x\| \leq \biggl( \int\limits_0^{\infty} \| A^{1/2} R
(e^{i \omega} t,A)x\|^2 dt \biggr)^{1/2} \leq C_3 \|x\|.
$$

\noindent d)
$A$ has an $H^{\infty}(\Sigma(\sigma))$--calculus for one (all) $\sigma \in
(\omega(A), \pi]$.
\end{theorem}

\begin{proof}
First we fix $\sigma$ and $\omega$ with $\omega(A) < \sigma < \omega \leq \pi$
and show that a) $\Longrightarrow$ b) $\Longrightarrow$ c). For fixed 
$\omega(A) < \omega < \sigma \leq \pi$ we have that c) $\Longrightarrow$ d)
$\Longrightarrow$ a). In the last part of the proof we show that b) holds for
all $\omega > \omega(A)$ if it holds for one such $\omega$. Together with the
implications proved before this will complete the proof.

a) $\Longrightarrow$ b)
For $y \in \cal{R}(A)$ we use
a well known representation of fractional powers of $A$ (cf.~\cite{Ko})
$$
A^{is- \frac{1}{2}} y = \frac{\sin (\pi (is - \frac{1}{2}))}{\pi}
\int\limits_0^{\infty} t^{is - \frac{1}{2}} (t+A)^{-1} y dt.
$$
For $x \in \cal{R}(A^{1/2}) \cap {\cal D}(A^{1/2})$ and $y = A^{1/2}x$, it follows that 
$$
A^{is}x = \frac{\cosh (\pi s)}{\pi} \int\limits_0^{\infty} t^{is} 
[t^{1/2} (t+A)^{-1} A^{1/2} x ] \frac{dt}{t}
$$
and if we replace $A$ by the sectorial operator $e^{-i \theta}A$ with
$\theta < \pi - \sigma$
\begin{equation}
\frac{\pi}{\cosh (\pi s)} e^{\theta s} A^{is}x = \int\limits_0^{\infty}
t^{is} [e^{i \frac{\theta}{2}} t^{1/2} A^{1/2} (e^{i \theta} t+A)^{-1}x]
\frac{dt}{t}.
\end{equation}
With the substitution $t = e^u$ we get
$$
\frac{\pi e^{\theta s}}{\cosh (\pi s)} A^{is} x = \int\limits_0^{\infty} 
e^{ius} [e^{i \frac{\theta}{2}} e^{u/2} (e^{i \theta}e^u +A)^{-1} 
A^{1/2} x] du.
$$
Since $\cosh (\pi s) \sim e^{\pi|s|}$ the left hand side is in $L_2(\R,H)$
and the Plancherel identity for $H$--valued functions cf.~(S3) gives
\begin{equation*}
\begin{split}
\biggl(\int\limits_0^{\infty}\|A^{1/2}(e^{i \theta}t+A)^{-1} x \|^2 &dt 
\biggr)^{1/2} = \biggl( \int\limits_{- \infty}^{\infty} \| e^{u/2} 
(e^{i \theta}e^u +A)^{-1}A^{1/2} x\|^2 du \biggr)^{1/2} \\
& \leq \frac{\pi}{\sqrt{2 \pi}} \biggl( \int\limits_{- \infty}^{\infty}
\biggl| \biggl| \frac{e^{\theta s}}{\cosh (\pi s)} A^{is} x 
\biggr| \biggr|^2 ds \biggr)^{1/2} \leq C_2 \|x\|
\end{split}
\end{equation*}
by the assumption on $A^{is}$.
For $\omega = \pi - \theta$ we have $(e^{i \theta} t+A)^{-1} =
-R(e^{- i \omega}t,A)$ and the claim follows.

b) $\Longrightarrow$ c)
For $x \in \cal{R}(A) \cap {\cal D}(A)$ we have
$$
e^{i \omega} \int\limits_0^{\infty} AR(e^{i \omega}t,A)^2 x dt =
\lim_{r \to \infty} [- AR(e^{i \omega}t,A)x]_{\frac{1}{r}}^r = x.
$$
Hence for any $y \in H$ with $\Vert y \Vert \leq 1$
\begin{eqnarray*}
(x|y) & = & e^{i \omega} \int\limits_0^{\infty} (AR(e^{i \omega}t|A)^2x,y)dt \\
& = & e^{i \omega} \int\limits_0^{\infty} (A^{1/2} R(e^{i \omega}t,A)x|
(A^{\ast})^{1/2}R(e^{-i \omega}t,A^{\ast})y)dt
\end{eqnarray*}
and using (S1)
\begin{eqnarray*}
|(x|y)| & \leq & \biggl( \int\limits_0^{\infty} \| A^{1/2} R(e^{i \omega}
t,A)x \|^2 dt \biggr)^{1/2} \biggl( \int\limits_0^{\infty} \| (A^{\ast})^{1/2}
R(e^{- i \omega}t,A^{\ast})y\|^2 dt \biggr)^{1/2} \\
& \leq & \biggl( \int\limits_0^{\infty} \|A^{1/2} R(e^{i \omega} t,A)x\|^2
dt \biggr)^{1/2} C_2 \|y\|.
\end{eqnarray*}
Together with b) we obtain
$$
\frac{1}{C_2} \|x\| \leq \biggl| \biggl| \biggl( \int\limits_0^{\infty}
\|A^{1/2} R(e^{i \omega} t,A)x \|^2 dt \biggr)^{1/2} \biggr| \biggr|
\leq C_2 \|x\|.
$$

c) $\Longrightarrow$ d)
Choose $\sigma$ and $\omega$ with $\omega(A) < \omega < \sigma$. 
For $f \in H_0^{\infty}(\Sigma(\sigma))$ and $| \arg \mu| = \omega$ 
we calculate using the resolvent equation
\begin{equation*}
\begin{split}
A^{1/2} R(\mu,A)f(A) & = \frac{1}{2 \pi i} \int\limits_{\partial
\Sigma(\omega)} f(\lambda)A^{1/2}R(\mu, A)R(\lambda,A) d \lambda \\
& = \biggl( \frac{1}{2 \pi i} PV \int\limits_{\partial \Sigma(\omega)} 
\frac{f(\lambda)}{\lambda - \mu} d \lambda \biggr) A^{1/2} R(\mu, A) \\
&\quad - \frac{1}{2 \pi i}PV \int\limits_{\partial \Sigma(\omega)} \frac{f(\lambda)A^{1/2}R(\lambda,A)}
{\lambda - \mu} d \lambda \\
& = f(\mu) A^{1/2} R(\mu, A) - K[f(\cdot) A^{1/2}R(\cdot,A)](\mu)
\end{split}
\end{equation*}
by Cauchy's theorem and using the notation
$$
KG(\mu) = \frac{1}{2 \pi i} PV \int\limits_{\partial \Sigma(\omega)} 
\frac{G(\lambda)}{\lambda - \mu} d \lambda, \quad \mu \in \partial 
\Sigma(\omega).
$$
$K$ is a variant of the Hilbert transform and a bounded operator on
$L_2(\partial \Sigma(\omega))$. By (S3) we have
for $g \in L_2(\partial \Sigma(\omega),H)$
\begin{equation*}
\int\limits_{\partial \Sigma(\omega)}\|Kg(\mu)\|^2 d|\mu| \leq
M^2 \int\limits_{\partial \Sigma(\omega)} \|g(\lambda) \|^2 d|\lambda|.
\end{equation*}
From this and c), it follows for $x \in \cal{D}(A)$ that
\begin{equation*}
\begin{split}
 \| f(A)x\|_H &\leq C_3 \|A^{1/2} R(\cdot,A) [f(A)x] \|_{L_2(\partial
\Sigma(\omega),H)} \\
& \leq C_3 \| f(\cdot) A^{1/2} R(\cdot,A)x\|_{L_2(\partial \Sigma(\omega),H)} \\
&\quad + C_3 M \| f(\cdot) A^{1/2} R(\cdot, A) x \|_{L_2(\partial \Sigma(\omega),H)} \\
& \leq C_3 (1+M) \|f\|_{H^{\infty}(\partial \Sigma(\omega))} \|A^{1/2} R
(\cdot, A)x\|_{L_2(\partial \Sigma(\omega),H)} \\
& \leq C_4 \|x\| \|f\|_{\infty}.
\end{split}
\end{equation*}

d) $\Longrightarrow$ a)
is clear since $\sup \{ |\lambda^{it} | : \lambda \in \Sigma(\sigma) \} \leq
e^{\sigma |t|}$.

Finally we show that we can choose $\omega > \omega(A)$ freely.
Consider $\nu$ and $\omega$ with $|\nu|, |\omega| > \omega(A)$.
The resolvent equation implies that
\begin{equation}
A^{1/2} R(te^{i \omega},A) = [I + (e^{i \nu} - e^{i \omega})tR(t e^{i \omega},
A)] A^{1/2} R(t e^{i \nu},A).
\end{equation}
Since the factor in square brackets is bounded on $\R_{+}$, we obtain
$$
\int\limits_0^{\infty} \|A^{1/2} R(t e^{i \omega},A)x\|^2dt \leq C
\int\limits_0^{\infty}\| A^{1/2} R(t e^{i \nu},A)x \|^2 dt.
$$
The same argument applies to $A^{\ast}$.

\end{proof}

Boyadzhiev and deLaubenfels have shown in \cite{BD} that a group generator on a
Hilbert space has an $H^{\infty}$--calculus on a strip. We give a proof using
square functions which is a variant of the proof
of Haase \cite{Ha2}. (Condition c) may be new.)

\begin{theorem}
For an operator $A$ of strip--type on a Hilbert space, the following conditions
are equivalent:

\noindent a)
$A$ generates a strongly continuous group $T_t$ on $H$.

\noindent b)
There is a constant $C < \infty$ so that for one (all) $b > \omega(A)$
\begin{eqnarray*}
& & \biggl( \int\limits_{- \infty}^{\infty} \|R(\pm b + it,A)x\|^2 dt 
\biggr)^{1/2} \leq C \|x\|, \\
& & \hspace*{8cm} \mbox{ for } x \in H \\
& & \biggl( \int\limits_{- \infty}^{\infty} \|R(\pm b+it,A)^{\ast}x\|^2 dt
\biggr)^{1/2} \leq C \|x\| .
\end{eqnarray*}

\noindent c)
There is a constant $C < \infty$ so that for one (all) $b > \omega(A)$
$$
\frac{1}{C} \|x\| \leq \biggl( \int\limits_{\partial S(b)} \|R(\lambda,A)
x\|^2d |\lambda| \biggr)^{1/2} \leq C \|x\|, \quad x \in H.
$$

\noindent d)
$A$ has an $H^{\infty}(S(a))$--functional calculus for one (all) $a > 
w(A)$.

Furthermore, we have $w_{H^{\infty}}(A) = w(A) = \omega(T_t)$.
\end{theorem}

\begin{proof}
a) $\Longrightarrow$ b)
Since e.~g.~$R(b +it,A)x = \int\limits_0^{\infty} e^{-(b+it)s} T_s x ds$
for $b > \omega(T_t)$ the claim follows from the vector--valued Plancherel
identity. If we have condition b) for one $b > \omega(A)$ we obtain it
for every other $\beta > w(A)$ using the resolvent equation. Since
$$
R(\beta + it,A) = [ I + (b - \beta)R(\beta + it,A)]R(b + it,A)
$$
we get
\begin{multline*}
\|R(\beta + i \cdot ,A)x\|_{L_2(H)} \leq [1+|b- \beta| \sup_t \|R(\beta
+ it,A)\|] \|R(b+i \cdot,A)x\|_{L_2(H)}.
\end{multline*}

b) $\Longrightarrow$ c)
Let $\varrho_n \in H_0^{\infty}(S(\alpha))$ with $\| \varrho_n \| \leq 1,
\varrho_n (\lambda) \to 1$ and $\varrho \in H_0^{\infty}(S(\alpha))$ with
$\cal{R}(\varrho(A))$ dense. Then for $x \in \cal{R}(\varrho(A))$ the convergence lemma
implies $\displaystyle \lim_n \varrho_n(A)x = x$ and from
$$
\varrho_n (A)x = \frac{1}{2 \pi i} \int\limits_{\partial S(\alpha)}
\varrho_n(\lambda) R(\lambda,A)x d \lambda
$$
we obtain for $n \to \infty$
\begin{equation*}
\begin{split}
x & = \frac{1}{2 \pi i} \int\limits_{\partial S(\alpha)} R(\lambda,A)x d 
\lambda 
= \frac{1}{2 \pi} \int\limits_{- \infty}^{\infty} [R(\alpha +it,A)x -
R(- \alpha +it,A)x]dt \\
& = -\frac{\alpha}{\pi} \int\limits_{- \infty}^{\infty} R(\alpha +it,A)
R(- \alpha +it,A)x dt.
\end{split}
\end{equation*}
Hence for $y \in H$ using (S2)
\begin{equation*}
\begin{split}
|(y|x)| & \leq \frac{\alpha}{\pi} \int_{-\infty}^\infty (R(\alpha +it,A)^{\ast}y| R(-\alpha
+it,A)x) dt \\
& \leq \frac{\alpha}{\pi} \|R(\alpha + i \cdot,A)^{\ast}y\|_{L_2(H)} 
\cdot \|R(- \alpha + i \cdot,A)x\|_{L_2(H)} \\
& \leq \frac{\alpha}{\pi} C \|R(- \alpha + i \cdot,A)x\|_{L_2(H)} \|y\|.
\end{split}
\end{equation*}
Since $y \in H$ was arbitrary the lower estimate of c) follows.

We can get from one $b > w(A)$ using the same trick as in a)
$\Longrightarrow$ b).

c) $\Longrightarrow$ d)
Let $a > b$, where $b$ is as in c). For $f \in H_0^{\infty}(S(a))$
we observe that for $\mu \in \partial S(b)$
\begin{equation*}
\begin{split}
R(\mu,A)f(A) & = \frac{1}{2 \pi i} PV\int\limits_{\partial S(b)}
f(\lambda)R(\mu,A)R(\lambda,A)d \lambda \\
& = \frac{1}{2 \pi i} PV\int\limits_{\partial S(b)} \frac{f(\lambda)}
{\lambda - \mu} d\lambda R(\mu, A) + \frac{1}{2 \pi i} 
PV\int\limits_{\partial S(b)} \frac{f(\lambda)R(\lambda,A)}{\lambda - \mu}
d \lambda \\
& = f(\mu) R(\mu,A) + \frac{1}{2i} H_0[f(\cdot) R(\cdot,A)](\mu)
\end{split}
\end{equation*}
where $H_0$ is the Hilbert transform on $L_2(\partial S(b))$. If ${\cal H}$ is the extension
of $H_0$ to $L_2(\partial S(b),H)$, we obtain with c), (S2) and (S3)
\begin{equation*}
\begin{split}
\frac{1}{C} \|f(A)x\| & \leq C\|R(\cdot,A)f(A)x\|_{L_2(\partial S(b),H)} \\
& \leq \|f(\cdot)R(\cdot,A)x\|_{L_2(\partial S(b),H)} +
\frac{1}{2} \| {\cal H} \| \|f(\cdot)R(\cdot,A)x\|_{L_2(\partial S(b),H)} \\
& \leq \|f\|_{H^{\infty}(S(a))} \left(1+\frac{1}{2}\|H_0\|\right) \|R(\cdot,A)x\|_{L_2(\partial 
S(b),H)} \\
& \leq C \|f\|_{H^{\infty}(S(a))}\|x\|.
\end{split}
\end{equation*}

d) $\Longrightarrow$ a)
Since $\sup \{ |e^{t \lambda}|: \lambda \in S(a) \} \leq e^{a|t|}$ the
$H^{\infty}$--calculus implies that $A$ generates a group. The strong
continuity follows from the convergence lemma.
\end{proof}

\begin{remark}
Following the argument $b) \Rightarrow c) \Rightarrow d)$, we obtain the estimate
\[ \| f(A) \|  \leq \frac{2C^2}{3\pi}a \cdot \| f \|_{H^\infty(S(a))}, \]

where $C$ is the constant in condition b).
\end{remark}

\section{Square functions in $L_q$--spaces}

Next we consider the classical square functions known from
harmonic analysis. 

Let $(\Omega, \mu)$ be a $\sigma$--finite measure space and $I \subset \R$
an interval. A function $f:I \to L_q(\Omega, \mu), 1 < q < \infty$, may be
viewed as a measurable function $\tilde{f}$ on $I \times \Omega$ with
$\tilde{f}(t, \omega) = f(t)(\omega)$ a.~e.~(by Fubini's theorem, 
cf.~\cite{DS}, Sect.~III.1.1.). For this function $\tilde{f}$ we may ask whether
it is finite with respect to the norm $L_q(\Omega, L_2(I))$ which we will
call the {\sl square function norm} of $f$ and denote by
\begin{equation}
\|f\|_{L_q(\Omega, L_2(I))} = \biggl|\biggl| \biggl( \int\limits_I |f(t)(\cdot)|^2dt 
\biggr)^{1/2}
\biggr|\biggr|_{L_q(\Omega)} \in [0, \infty].
\end{equation}
By Fubini's theorem we have of course that $L_2(\Omega,L_2(I)) = L_2
(I,L_2(\Omega))$. Hence this norm extends the square functions considered in
Section 2. 

We will be interested in functions of the form $f(t) =
A^{1/2} R(-t,A)x$ or $f(t) = t^{n - \frac{1}{2}}A^nT_tx$ on $\R_{+}$.
If $A = - \Delta$ or $A = (- \Delta)^{1/2}$ on $L_q(\R^N)$ one obtains the 
classical Paley--Littlewood $g$--functions (see e.~g.~\cite{St}). In \cite{CDMY} such
square functions were used for the first time for a general approach to the
$H^{\infty}$--calculus in $L_q(\Omega)$ spaces. \\
To make sure that such expressions make sense, at least for $x$ in a dense 
subspace of $X$,
one may use the following observation:

\begin{remark}
Assume that $f:[0,b] \to L_q(\Omega)$ is continuously differentiable.
Then
$$
\|f\|_{L_q(\Omega, L_2(I))} \leq \int\limits_0^b s^{1/2} \| f'(s)\|_{L_q(\Omega)} ds +
b^{1/2} \|f(b)\|_{L_q(\Omega)}.
$$
Indeed, we write 

\[ f(t,\cdot) = f(b,\cdot)-\int_0^b \chi_{[0,s]}(t)f'(s,\cdot) ds \]

so that by taking the $L_q(L_2)$ norm with respect to the variable $t\in(0,b)$, we obtain the above inequality.
\end{remark}

\begin{example}
a) Let $A$ generate an exponentially stable semigroup $T_t$ on $X$ with $\omega(T_t) < 0$.
Then we have for $x \in \cal{D}(A)$
\begin{eqnarray*}
\|T_tx\|_{L_q(\Omega, L_2(I))} & \leq & C \|Ax\| \\
\|R(i \cdot, A)x\|_{L_q(\Omega, L_2(I))} & \leq & C \|Ax\|.
\end{eqnarray*}

b) For a sectorial operator $A$ on $X$ we have for $x \in \cal{D}(A) \cap \cal{R}(A)$ 
$$
\| A^{1/2} R(-t,A)x\|_{L_q(\Omega, L_2(I))} \leq C (\| Ax\| + \|A^{-1}x\| +
\|x\|).
$$
If $\omega(A) < \textstyle \frac{\pi}{2}$ then $A$ generates an analytic
semigroup $T_t$ and for $x \in \cal{D}(A) \cap \cal{R}(A)$
$$
\|t^{- 1/2} (tA)^nT_tx\|_{L_q(\Omega, L_2(I))} \leq C (\|Ax\| + \|A^{-1}x\| +
\|x\|).
$$

Hence such square functions are finite at least for $x$ in a dense subset $D$ such as ${\cal D}(A)$ or ${\cal D}(A) \cap \cal{R}(A)$. If one can now establish for $x\in D$ better estimates of the form $\|A^{1/2}R(-t,A)x\|_{L_q(\Omega, L_2(I))} \leq C\| x\|$, then one can extend the continuous embedding $x\in D \to Jx=AR(-t,A)x \in L_q(\Omega,L_2[0,b])$ to all of $L_q(\Omega)$.
\end{example}

\begin{proof}
a) Indeed in the first formula $t^{1/2} \| \textstyle \frac{d}{dt} T_tx\| = t^{1/2} \| T_t
(Ax)\|$ is integrable on $\R_{+}$ and $t^{1/2}T_tx \to 0$ for $ t \to \infty$.
So we can apply 3.1 The second estimate follows from the first, since
$R(i \cdot, Ax) = {\cal F}(T_{(\cdot)}x)$ by (S3) below.

b) For a sectorial operator $A$ it is well known that $\displaystyle
\sup_t \| t^{1/2}f(t)\| < \infty$ for $f(t) = A^{1/2}R(-t,A)x$.
Also 
\begin{eqnarray*}
\|t^{1/2}f'(t) \| & = & \| t^{1/2}R(-t,A)^2(A^{1/2}x)\| \lesssim t^{- 3/2}
\|A^{1/2}x\| \\
\|t^{1/2}f'(t)\| & = & \|t^{1/2}[tR(-t,A)^2(A^{- 1/2}x) - R(-t,A)
(A^{- 1/2} x)] \| \\
&\lesssim & t^{- 1/2} \|A^{- 1/2} x \|.
\end{eqnarray*}
If we use the first estimate for large $t$ and the second for small $t$
we see that $\|t^{1/2}f'(t)\|$ is integrable and 3.1 applies.

If $\omega(A) < \textstyle \frac{\pi}{2}$ then $\displaystyle 
\sup_{t > 0} \|t^nA^nT_t\| < \infty$ and we can apply 3.1 in a similar way.
\end{proof}

These square functions still share the basic properties (S1), (S2) and (S3)
of Section 2, but with one important difference:
In (S2) we have to replace the boundedness of $N(t)$ by $R$--boundedness.
More precisely

\begin{itemize}
\item[{(S1)}] 
If $f \in L_q(\Omega, L_2(I))$ and $g \in L_{q'} (\Omega, L_2(I))$ with
$\textstyle \frac{1}{q} + \textstyle \frac{1}{q'} = 1$, then
$$
| < f,g>| \leq \|f\|_{L_q(\Omega, L_2(I))} \|g\|_{L_{q'}(\Omega, L_2(I))}.
$$

\item[{(S2)}] 
If $t \in I \to N(t) \in B(L_q)$ is strongly continuous and $R$--bounded,
then
$$
\|N(t)f(\cdot)\|_{L_q(\Omega, L_2(I))} \leq R(N(t):t \in I) \|f(\cdot)
\|_{L_q(\Omega, L_2(I))}.
$$

\item[{(S3)}] 
For a bounded operator $S:L_2(I) \to L_2(I)$ put $\calS \biggl(
\displaystyle \sum_{j=1}^n
f_j \otimes x_j \biggr):= \displaystyle \sum_{j=1}^n S(f_j) \otimes x_j$ for $f_j 
\in L_2(I)$ and $x_j \in L_q(\Omega)$. Then $\calS$ extends to a bounded
operator $\calS : L_q(\Omega, L_2(I)) \to L_q(\Omega, L_2(I))$ with $\|\calS\| \leq
\|S\|_{L_2(I)}$.
\end{itemize}

(S1) is clear, (S2) is checked in \cite{W2}, and (S3) follows by direct
calculation from the definition. Recall that a sequence $T_j$ of operators
on $L_q(\Omega)$ is {\sl $R$--bounded} if and only if there is a constant $C
< \infty$ such that for all $(x_j) \subset L_q(\Omega)$
$$
\biggl| \biggl| \biggl( \sum_j | T_j x_j (\cdot) |^2 \biggr)^{1/2}
\biggr| \biggr|_{L_q(\Omega)} \leq C \biggl| \biggl| \biggl( \sum_j |x_j (\cdot)|^2
\biggr)^{1/2} \biggr| \biggr|_{L_q(\Omega)}.
$$

Having these properties in mind we could
easily extend the proofs of Section 2 to the $L_p$--case and recover the
characterization of the $H^{\infty}$--calculus in terms of square functions
from \cite{CDMY} or prove the following new results:

\begin{theorem}
A sectorial operator $A$ on $L_q(\Omega)$, $1 < q < \infty$, has an
$H^{\infty}$--calculus if and only if $A$ has BIP and $\{ e^{-a|t|} A^{it}:
t \in \R \}$ is $R$--bounded for some $a \in [0, \pi)$. Furthermore,
$\omega_{H^{\infty}}(A) = \omega_R(A^{it})$.
\end{theorem}

\begin{theorem}
An operator $A$ of strip--type on $L_q(\Omega), 1 < q < \infty$, has an
$H^{\infty}$--calculus on a strip if and only if $A$ generates a 
$C_0$--group such that $\{ e^{- b|t|}T_t:t \in \R\}$ is $R$--bounded for some
$b > 0$. Furthermore, $w_{H^{\infty}}(A) = \omega_R(T_t)$.
\end{theorem}

Instead of proving these results now we will first introduce generalized 
square functions on Banach spaces
and then present our proofs in this more general setting.
A reader premarily interested in the $L_p$--case, can easily interpret these
arguments in terms of the square functions defined above. Since $L_p$--spaces
with $1 < p < \infty$ have non--trivial type and cotype and property 
$(\alpha)$, no additional assumptions are necessary in this case.

To motivate our general definition of square function in the next section,
we reformulate (1):

\begin{discussion}
To a function $f \in L_q(\Omega, L_2(I)), \Omega \subset \R^N$, we can associate
an operator $u_f : L_2(I) \to L_q(\Omega)$ by 
$$
u_f(h) = \int\limits_I f(t) h(t) dt, \quad h \in L_2(I).
$$
If $(e_n)$ is an orthonormal basis of $L_2(I)$, then so is 
$(\overline{e}_n)$
and for a continuous $f$ on $I \times \Omega$, we calculate for a fixed $w$
\begin{equation*}
\begin{split}
& \biggl( \int\limits_I |f(t)(\omega)|^2 dt \biggr)^{1/2} =
\biggl( \sum_n |<f(\cdot)(\omega), \overline{e}_n > |^2 \biggr)^{1/2} \\
& = \biggl( \sum_n |[u_f(e_n)](\omega)|^2 \biggr)^{1/2} 
= \biggl( \E \biggl| \sum_n g_n [u_f(e_n)](\omega) \biggr|^2 \biggr)^{1/2} \\
& = C_q \biggl( \E \biggl| \sum_n g_n [u_f (e_n)](\omega) \biggr|^q \biggr)^{1/q}
\end{split}
\end{equation*}
where $(g_n)$ is a Gaussian sequence. Taking norms in
$L_q(\Omega)$ and using Fubini's theorem we obtain
\begin{align*}
\biggl| \biggl| \biggl( \int_I | f(t) |^2 dt \biggr)^{1/2} 
\biggr| \biggr|_{L_q(\Omega)} 
&= C_q \biggl( \E \biggl| \biggl| \sum_n g_n u_f (e_n) 
\biggr| \biggr|_{L_q(\Omega)}^q \biggr)^{1/q} \\
&\simeq C_q' \biggl( \E \biggl| \biggl| \sum_n g_n u_f (e_n) 
\biggr| \biggr|_{L_q(\Omega)}^2 \biggr)^{1/2}
\end{align*}
where we used Kahane's inequality for the last estimate. Note that the last expression makes sense
in every Banach space.
\end{discussion}

\begin{remark}
Let $X$ be a Banach function space on $(\Omega,\mu)$ which is $q$--concave
for some $q < \infty$ in the sense of \cite{LT}, Def.~1.d.3. Then we can replace
Fubini's theorem in the above calculation by a result of Maurey (see \cite{LT}, 
1.d.6) and still obtain that
$$
\biggl| \biggl| \biggl( \int_I |f(t)(\cdot)|^2dt \biggr)^{1/2} \biggr| 
\biggr|_X \simeq \E \biggl( \biggl| \biggl| \sum_n g_n u_f (e_n) 
\biggr| \biggr|^2 \biggr)^{1/2}.
$$
Hence in $q$--concave function spaces the classical square functions (they 
were used e.~g.~in \cite{LLL}) are also equivalent to the generalized square functions to be 
introduced in the next section.
\end{remark}

\section{Generalized square functions}

In this section $H$ is always a Hilbert space and denote by $(\cdot | \cdot)$
its scalar product. If $X$ is a Banach space, we will need on occasion the 
Banach space dual $u: X'\to H'$ of operators $u\in B(H,X)$. To avoid notational 
complications we fix a bilinear map $< \cdot , \cdot >_H$ on $H \times H$ so 
that $\|x\|_H = \sup\{ < x,y >_H: \|y\|_H \leq 1 \}$ and $\|y\| = \sup 
\{ < x,y >_H :\|x\|_H \leq 1 \}$, so that we can identify the Banach space dual $H'$ with
$H$ via this duality. If $(e_j)$ is an orthonormal basis of $H$ we can e.g. choose 
$$
<x,y> = \displaystyle \sum_j <x, e_j> <y, e_j>.
$$

Of course, $<\cdot , \cdot>$ is only determined up to unitary maps of $H$,
but if $H = L_2(\Omega, \mu)$ we always choose $<f,g> = \int f(\omega)
g(\omega) d\mu (\omega)$. For $S \in B(H,X)$ we denote by $S'$ the Banach
space dual $S' \in B(X', H)$ with respect to the duality $<\cdot , \cdot>_H$. For $S\in B(H)$, we also define $S' \in B(H)$ by $<Sh,g>=<h,S'g>$ and distinguish this adjoint from the Hilbert space adjoint $S^*$ which is determined by the scalar product $(\cdot\vert\cdot)$. If $X=L_2(\Omega,\mu)$,  then $S'h=\overline{S^{*}(\overline{h})}$.
$u= \sum x_i \otimes h_i$ with $x_i \in X$ and $h_i \in H$ stands for the
operator $u(h)= \sum_i <h, h_i>_H x_i$.

For our treatment of generalized square functions we need some notions from 
Banach space theory. Recall that $(g_n)$ is an independent sequence of 
standard Gaussian variables. $(g_n)$ is supposed to be complex if $H$ is a complex Hilbert space.

\begin{definition}
Let $H$ be a Hilbert space and $X$ a Banach space. We denote by 
$\gamma_{+}(H,X)$ the 
space of all linear operators $u: H \to X$ such that 
$$
\|u\|_\gamma = \sup_{N\in \mathbb{N}} \biggl( \E \biggl| \biggl| \sum_{n=1}^N g_n u(e_n) 
\biggr| \biggr|^2 \biggr)^{1/2} < \infty.
$$
Here the $\sup$ is taken over all finite orthonormal system $(e_n)$ in $H$.
\end{definition}

By $\gamma(H,X)$ we denote the closure of the finite dimensional operators in
$\gamma_{+}(H,X)$. The space $\gamma(H,X)$ and its norm play an important role in the geometry of Banach spaces (cf \cite{Pi5}, p.35ff) and in the theory of cylindrical Gaussian measures on Banach spaces (see e.g. \cite{LPie}). However, in this literature the notation $l(H,X)$ instead of $\gamma(H,X)$ is used.

\begin{remark}
a) If $X$ does not contain $c_0$, then $\gamma_{+}(H,X) = \gamma(H,X)$ by a result of
Kwapien \cite{Kw}. 

b) Let $H$ be separable and $(e_j)$ an orthonormal basis of $H$.
If $u \in \gamma(H,X)$ then
$$
\|u\|_\gamma = \biggl( \E \biggl| \biggl| \sum_{j \in \N} g_j u(e_j)
\biggr| \biggr|^2 \biggr)^{1/2}.
$$

c) By Kahane's inequality (\cite{LeTa}), one can define on $\gamma_{+}(H,X)$ equivalent norms
\[ \|u\|_{\gamma,p} = \sup_{N\in\mathbb{N}}\biggl( \E \biggl| \biggl| \sum_{j=1}^N g_j u(e_j)
\biggr| \biggr|^p \biggr)^{1/p} \]
 
for $1<p<\infty$.
\end{remark}

This expression is independent of the choice of $(e_j)$ thanks to
the following ideal property of $\gamma(H,X)$ applied to unitary operators on $H$:

\begin{proposition}
For $T \in B(X,Y), u \in \gamma(H_2,X)$ and $v \in B(H_1,H_2)$ the composition $T 
uv \in \gamma(H_1,Y)$ and 
$$
\| Tuv\|_\gamma \leq \|T\| \cdot \|u\|_\gamma \cdot \|v\|.
$$
\end{proposition}

\begin{proof}
For a finite orthonormal system $(e_n)_{n=1}^m$ in $H_1$ choose a finite 
orthonormal system $(f_n)_{n=1}^m$ in $H_2$, so that $v(e_i) \in \mbox{ span 
}(f_1,\dots,f_m)$ for $i=1,\dots,n$. Then there are 
$(a_{ij})_{i,j=1,\dots,m}$ so that $Pv^{\ast}f_i=\displaystyle \sum_{j=1}^m 
a_{ij}e_j$, 
where $P$ is the orthogonal projection onto span $(e_1,\dots,e_m)$. Then 
$$
uv(e_n) =u \biggl(\sum_{i=1}^m (f_i|ve_n)f_i \biggr) = \sum_{i=1}^m 
a_{in}u(f_i).
$$
Hence by Corollary 12.17 in \cite{DJT} (or by the complex version) we obtain
\begin{equation*}
\begin{split}
& \E \biggl| \biggl| \sum_n g_n Tuv(e_n) \biggr| \biggr|^2 \leq \| T\|^2 \E 
\biggl| \biggl| \sum_n g_n \biggl( \sum_{i=1}^n a_{in} uf_i \biggr) \biggr| 
\biggr|^2 \\
& \leq \|T\|^2 \cdot \| Pv^{\ast}\|^2 \E \biggl| \biggl| \sum_n g_n 
u(f_n) \biggr| \biggr|^2 \leq \|T\|^2 \|v\|^2 \|u\|_\gamma^2.
\end{split}
\end{equation*}
\end{proof}

We would like to extend an operator $S \in {\cal B}(H_1,H_2)$ to an operator 
$\Skr:\gamma(H_1,X) \to \gamma(H_2,X)$, so that for a finite dimensional 
operator $u = \sum x_i \otimes h_i$ with $x_i \in X, h_i \in H_1$ we have 
$$
{\bf S}^{\otimes}(u) = \sum x_i \otimes Sh_i
$$
and hence for $g \in H_2$
$$
\Skr(u)(g) = \sum x_i <S'g,h_i>=(u \circ S')(g)
$$
This is possible by 4.3.

\begin{proposition}
Let $S \in {\cal B}(H_1,H_2)$. Then the operator 
$$
{\bf S}^{\otimes} : \gamma(H_1,X) \to \gamma(H_2,X), \quad 
{\bf S}^{\otimes}(u) = u \circ S'
$$
defines a bounded operator with $\|{\bf S}^{\otimes}\| \leq \|S\|$.
Moreover, for all $x' \in X'$ and $u \in \gamma(H_1,X)$
$$
<{\bf S}^{\otimes}u,x'>_X = S (<u,x'>_X)
$$
as functionals on $H_2$.
\end{proposition}

\begin{proof}
The norm estimate follows from the ideal property 4.3. We have for $u \in \gamma(H_1,X)$ $x' \in X'$ and $h \in H_2$
\begin{align*}
[<{\bf S}^{\otimes}&u,x'>_X](h)  =  < {\Skr}u(h),x'>_X = <u(S'h),x'>_X \\
& =  < S'h,u'(x')>_H = <h, S(u'(x'))>_H = [S (<u,x'>_X)](h).
\end{align*}
\end{proof}

Let $(\Omega, \Sigma, \mu)$ be a $\sigma$--finite measure space and $X$ a 
Banach space. By $\mathcal{P}_2(\Omega,\mu,X)$ we denote all 
Bochner--measurable functions $f:\Omega \to X$, so that $<f,x'>_X \in 
L_2(\Omega, \mu)$ for all $x' \in X'$. For $f \in \mathcal{P}_2(\Omega, \mu, 
X)$ 
we can define an operator $u_f \in B(L_2(\Omega, \mu),X)$ so that for $h \in 
L_2(\Omega, \mu)$ and $x' \in X'$ 
$$
<u_fh,x' >_X = \int\limits_{\Omega} <f(\omega),x'>_X h(\omega) d\mu(\omega).
$$
(By the uniform boundedness principle this formula defines a bounded operator $u_f 
: H \to X^{''}$. If $f$ is bounded on $\operatorname{supp}h$ and $\mu(\operatorname{supp}h)<\infty$, then the Bochner 
integral $u_fh = \int f(\omega)h(\omega)d\mu(\omega)$ belongs to $X$. Since 
such $h$ are dense in $L_2(\Omega, \mu)$ we conclude that the range of $u_f$ is 
contained in $X$.)
Now we can introduce our square function norm.

\begin{definition}
Let $f \in \mathcal{P}_2(\Omega,X)$ and $u_f$ as above. Then if $u_f \in$ \allowbreak 
$\gamma_{+}(L_2(\Omega),X)$ we define
$$
\|f\|_{\gamma(\Omega,X)} := \|u_f\|_{\gamma(L_2(\Omega),X)}.
$$
The space of all $f$ for which $u_f \in \gamma_{+}(L_2(\Omega),X)$ (or $u_f \in
\gamma(L_2(\Omega),X))$ we denote by 
$\gamma_{+}(\Omega,\mu,X)$ (or $\gamma(\Omega,\mu,X)$).
\end{definition}

If $\Omega = \Z$ with the counting measure then for $x_i = f(i), i \in \Z$,
we also use the notation
$$
\|(x_i)_i\|_\gamma = \|u_f\|_{\gamma(\Z,X)} = \sup_{N\in \mathbb{N}}\biggl( \E \biggl| \biggl| \sum_{n=1}^N
g_nx_n \biggr| \biggr|^2 \biggr)^{1/2}.
$$

\begin{examples}
a) If $f= \displaystyle \sum_{i=1}^n x_i \chi_{A_i}$ is a step function with 
$x_i 
\in 
X$ 
and $\mu(A_i) < \infty$ and $A_i$ pairwise disjoint, then 
$$
\|f\|_{\gamma(\Omega,X)} = \biggl( \E \left|\left| \sum_{i=1}^n g_i \mu(A_i)^{1/2} x_i\right|\right| ^2
\biggr)^{1/2} = \|(\mu(A_i)^{1/2}x_i)_i\|_\gamma.
$$
Indeed, if we choose the basis $h_i = \mu(A_i)^{- 1/2} \chi_{A_i}$ in span 
$(\chi_{A_i})$ and complete it by an orthonormal basis of $\{ \chi_{A_i} 
\}^{\perp}$ then $u_f(h_i) = \mu (A_i)^{1/2} x_i$.

b) Suppose that $f:[0,b] \to X$ is continuously differentiable. Then $f \in 
\gamma([0,b],X)$ and
$$
\|f\|_{\gamma([0,b],X)} \leq \int\limits_0^b s^{1/2} \|f'(s)\|_X ds + 
b^{1/2}\|f(b)\|_X.
$$
(The proof is the same as in 4.2 Just note that $\| - 
\chi_{[0,s]} f'(s)\|_\gamma = s^{1/2}\|f'(s)\|$.)
\end{examples}

\begin{remark}
a) The set $u_f,f \in \gamma(\Omega,X)$, is 
a dense and in general proper subspace of $\gamma(L_2(\Omega),X)$. Therefore, one may consider the 
space of operators $\gamma(L_2(\Omega),X)$ as the completion of the function 
space $\gamma(\Omega,X)$.

b) If $X$ has type 2, then $L_2(\Omega,X) \subset \gamma(\Omega,X)$.
If $X$ has cotype 2, then $\gamma(\Omega,X) \subset L_2(\Omega,X)$.
(Use the definition of type and cotype and Lemma 4.10 below.)

c) The rapidly decreasing functions $\calS(\R^N,X)$ belong to $\gamma(\R^N,X)$.
(For $N=1$ this follows from 4.6b) for $b \to 0$ and $\|f\|_\gamma \leq \int\limits_{- \infty}^{\infty} |t|^{1/2} \|f'(t)\| dt$. This argument can be extended to the multidimensional case. For general $N$ see \cite{KW2}.)
\end{remark}

We identify the dual of $H_j=L_2(\Omega_j,\mu)$, $j = 1,2$, with 
$L_2(\Omega_j,\mu)$ via the duality map $<f,g> = \int 
f(\omega)g(\omega)d\mu(\omega)$. 
Since for $g \in \mathcal{P}_2(\Omega_1), h \in H_2$ and $x' \in X'$
$$
<u_g(h), x'>_X = \int <f(\omega), x'>_X h(\omega) d \mu(\omega)
$$
the extension principle 4.4 gives in the case of functions

\begin{corollary}
For $S \in B(L_2(\Omega_1),L_2(\Omega_2))$ let $\Skr$ be the extension
of $S$ from $\gamma(L_2(\Omega_1),X)$ to $\gamma(L_2(\Omega_2),X)$. If 
$f \in \gamma(\Omega_1,X)$ and there is $g \in \mathcal{P}_2(\Omega_2,X)$  
with $\Skr(u_f)=u_g$, then $\|g\|_{\gamma(\Omega_2,X)} \leq \|S\| \cdot 
\|f\|_{\gamma(\Omega_1,X)}$ and $<g,x'>_X = S(<f,x'>_X)$ for $x' \in X'$.
\end{corollary}

{\sl Warning:}
It may happen that $\Skr(u_f)$ for $f \in \gamma(\Omega_1,X)$ does not belong 
to $\gamma(\Omega_2,X)$, but only to its completion $\gamma(L_2(\Omega_2),X)$. 

The following examples show that 4.8 will play the role of (S3) in our 
work with square functions. It provides a natural extension of "kernel" operators on $L_2(\Omega)$ to the function space $\gamma(\Omega,X)$.

\begin{example}
a) (Multiplication with scalar functions)
Let $m \in L_{\infty}(\Omega)$. The multiplication operator $Mf(t) =
m(t)f(t)$ is bounded on $L_2(\Omega)$. For $f\in\mathcal{P}_2(\Omega,X)$, we have $Mf \in \mathcal{P}_2(\Omega,X)$ and $\Mkr u_f = u_{Mf}$ by 4.8. Hence the definition of $M$ extends to $\gamma(\Omega,X)$ and $\| Mf \|_{\gamma} \leq \| m \|_{L_\infty} \| f \|_\gamma$.

b) (Fourier transform)
The Fourier transform on $L_2(\R^N)$ is defined by 
\[ {\cal F}f(t)=(2 \pi)^{- N/2} \int\limits_{\R^N} e^{-its}f(s)ds \] 
for $f\in L_2(\R^N) \cap L_1(\R^N)$. For $f\in \mathcal{P}_2(\R^N,X) \cap L_1(\R^N,X)$, the same formula gives ${\cal F}f \in \mathcal{P}_2(\R^N,X)$. By 4.8 we have ${\bf \mathcal{F}}^{\otimes}u_f = u_{\mathcal{F}f}$. Hence, the definition extends to (a dense set of functions in) $\gamma (\R^N,X)$ and 
\[ \| \mathcal{F}f \|_{\gamma(\R^N,X)} = \| f \|_{\gamma(\R^N,X)} \] 
by 4.8. Of course this extension can also be applied to other integral transforms 
bounded on $L_2$ such as the Mellin transform or the Hilbert transform.

c) (Integral operators)
Let $k$ be a measurable kernel on $\Omega_2 \times \Omega_1$, so that 
$k(t,\cdot)f(\cdot)$ is integrable on $\Omega_1$ for $f \in L_2(\Omega_1)$ 
and almost all $t \in \Omega_2$ and 
$$
Kf(t) = \int\limits_{\Omega_1} k(t,s) f(s)ds, \quad t \in \Omega_2
$$
defines a bounded operator $K:L_2(\Omega_1) \to L_2(\Omega_2)$. For $f\in L_2(\Omega_1,X)$, the same formula defines a function $Kf \in L_2(\Omega_2,X)$. By 4.8, $\Kkr u_f = u_{Kf}$ and $\|Kf\|_{\gamma(\Omega_2,X)} \leq \|K\| \|f\|_{\gamma(\Omega_1,X)}$.

d) (Averaging projections)
Let $(\Omega,\Sigma,\mu)$ be a $\sigma$--finite measure space and 
$(E_k)_k$ a (finite or infinite) partition of $\Omega$ with $0 < \mu(E_k)
< \infty$. Then
$$
Ph = \sum_k \mu(E_k)^{-1} \int\limits_{E_k} h d\mu \; \chi_{E_k}
$$
defines a contraction on $L_2(\Omega)$. Again, the same formula defines an operator on $L_2(\Omega,X)$, which extends to a contraction
$\Pkr$ on $\gamma(L_2(\Omega),X)$. For $f \in L_2(\Omega,X) \cap \gamma(L_2(\Omega),X)$, we have $\Pkr u_f = u_{Pf}$.
\end{example}

The following convergence results are useful:

\begin{lemma}
a) Suppose $u_{\nu}$ is a uniformly bounded net in ${\cal B}(X)$ such
that $\displaystyle \lim_{\nu} u_{\nu} = u$ in the strong operator topology.
Then $\|u\|_\gamma \leq \displaystyle \liminf_{\nu} \|u_{\nu}\|_\gamma$.

b) Suppose that $f_n,f \in {\cal P}_2(\Omega,X)$ and $f_n \to
f$ as $n\to\infty$ $\mu$--a.e. Then $\|f\|_{\gamma(\Omega,X)}$ $\leq
\liminf_n \|f_n\|_{\gamma(\Omega,X)}$.

c) Suppose $f \in \gamma(\Omega,X)$. Let $h_n \in L_{\infty}(\Omega)$ with
$\displaystyle \sup_{n} \|h_n\|_{\infty} < \infty$ and $h_n \to 0$ as $n\to\infty$ a.e. Then 
$\|h_nf\|_{\gamma(\Omega,X)} \to 0$ as $n\to\infty$.
\end{lemma}

\begin{proof}
a) For every orthonormal system $e_1, \dots, e_m$ we have by Fatou's lemma
\[ \|(u(e_j))_{j=1,\dots,m}\|_\gamma \leq \displaystyle
\liminf_{\nu} \|(u_{\nu}(e_j))_{j=1,\dots,m}\|_\gamma \].

b) For every $m$ there is by Egoroff's theorem an $\Omega_m \subset \Omega$
with $\mu(\Omega \setminus \Omega_m) \leq \textstyle \frac{1}{m}$ so that
$f$ is bounded on $\Omega_m$ and $f_n \to f$ uniformly on $\Omega_m$ as $n\to\infty$.
Now define the net $u_{n,m}$ by the functions $f_n \chi_{\Omega_m}$ and
apply a).

c) is clear if the range of $f$ is finite--dimensional. Since $f \in
\gamma(\Omega,X)$ we can use now an approximation argument.
\end{proof}

For the sake of simplicity we formulate the next proposition for a locally 
compact metric space $\Omega$ with no isolated points and a positive Borel 
measure $\mu$, i.~e.~$\mu(U) > 0$ for every open $U \subset \Omega$. This result will play the role of (S2).

\begin{proposition}
Let $N:\Omega \to B(X)$ be a strongly continuous map. Then $\tau = \{ N(t):t 
\in \Omega \}$ is $\gamma$--bounded with constant $K$ if and only if 
$$
\| N(\cdot)[f(\cdot)]\|_{\gamma(\Omega,X)} \leq K \|f\|_{\gamma(\Omega,X)}
$$
for all $f \in \gamma(\Omega, X)$.
\end{proposition}

\begin{proof}
Suppose first that $\tau$ is $\gamma$--bounded with constant $K$. Choose a 
partition $(E_k)_k$ of $\Omega$, so that $N$ and $f$ are bounded on each $E_k$ 
and $\mu(E_k) < \infty$. Then there are $x_k \in X$ and $S_k \in B(X)$ so 
that for the averaging projection $P$ from 4.9d)
$$
P(f) = \sum_k x_k \chi_{E_k}, \quad P(N) = \sum_k S_k \chi_{E_k}, \quad P(N)[P(f)] = 
\sum_k S_k x_k \chi_{E_k}.
$$
Note that $S_kx=\mu(E_k)^{-1} \int\limits_{E_k} N(\omega) x d\mu (\omega)$ 
belongs to the closure of $\mbox{absco}(\tau)$ in the strong operator topology, 
so that $\{ S_k\}$ is $\gamma$--bounded with the same constant $K$. Hence by 4.6 
and 4.9.
\begin{align*}
\| PN [Pf]\|_{\gamma(\Omega,X)} &= \|(\mu(E_k)^{1/2} S_k x_k )_k\|_\gamma \leq K \|(\mu(E_k)^{1/2}x_k)_k \|_\gamma \\
&= K \|Pf\|_{\gamma(\Omega,X)} \leq K \|f\|_{\gamma(\Omega,X)}.
\end{align*}
If we choose a sequence $P_n$ of averaging projections, so that $P_nf \to f$,
$P_n(N) \to N$ a.~e.~for $n \to \infty$, then the claim follows from 4.10.

For the converse we pick $N_k = N(\omega_k)$, $k = 1,\dots,n$, with distinct \allowbreak
$\omega_1,\ldots,\omega_n \in \Omega$ and $x_1,\dots,x_n \in X$. For a given $\varepsilon > 0$ choose
disjoint open sets $U_1,\dots,U_n$ with $\omega_k \in U_k$ and such that
$\|N(\omega)x_k - N(\omega_k)x_k \| \leq \varepsilon/n$ for $\omega \in U_k$.
Choose $f= \sum_k x_k \mu(U_k)^{- 1/2} \chi_{U_k}$ so that $\| f \|_\gamma = \| (x_k)_k \|_\gamma$ by example 4.6. If we put
$$
Ph = \sum_{k=1}^n \left( \mu(U_k)^{-1} \int\limits_{U_k} h d\mu \right) \chi_{U_k}
$$
then $P(Nf) = \displaystyle \sum_{k=1}^n \mu(U_k)^{-1/2} y_k \chi_{U_k}$ with $y_k = \mu(U_k)^{-1} \int_{U_k} N(\omega)x_k d\mu$ and
$$
\|(N_kx_k)_k - (y_k)_k \|_\gamma \leq \sum_{k=1}^n \|N_kx_k-y_k\|_X \mathbb{E}\vert g_k \vert \leq \varepsilon
$$
By 4.9 and our assumption
\begin{align*}
\|(N_kx_k)_k\|_\gamma &\leq \|P(Nf)\|_{\gamma(\Omega,X)} + \|(y_k)_k - (N_kx_k)_k\|_\gamma \\
& \leq \|Nf\|_{\gamma(\Omega,X)} + \varepsilon \leq K \|f\|_{\gamma(\Omega,X)} +
\varepsilon =  K \|(x_k)\|_\gamma + \varepsilon.
\end{align*}
Since $\varepsilon > 0$ was arbitrary we obtain the claim for distinct
$\omega_1,\dots,\omega_n$. For general $\omega_k$ we employ a limiting
argument where we approximate the given $\omega_1,\ldots,\omega_n$ by $n$ different points $\omega'_1,\ldots,\omega'_n$.
\end{proof}

Finally we derive a multiplier theorem from 4.11. Let $t \in \R^N \setminus
\{0\} \to N(t) \in B(X,Y)$ be strongly continuous and bounded. For every $f \in
\calS(\R^N,X)$, we have that $N(\cdot)[ \hat{f}(\cdot)]$ is in $L_1
(\R^N,X) \cap L_2(\R^N,X)$. Hence we have a map
$$
{\cal K}_N : \calS(\R^N,X) \to L_{\infty}(\R^N,Y), \quad {\cal K}_N f =
(N(\cdot)[f(\cdot)]\spcheck)\sphat.
$$

\begin{corollary}
${\cal K}_N$ extends to a bounded operator ${\cal K}_N : \gamma(L_2(\R^N),X) \to 
\gamma(L_2(\R^N),Y)$ if and only if $\{ N(t) : t \in \R^N \setminus \{0\} \}$ is
$R$--bounded.
\end{corollary}

\begin{proof}
Combine 4.11 with the identity
$$
\|\hat{f}\|_{\gamma(\R^N,X)} = \|f\|_{\gamma(\R^N,X)} \quad \mbox{ cf.~4.10a)}.
$$
\end{proof}

One can give these multiplier theorems a more general form by considering
strongly measurable functions $N(\cdot)$ and its Lebesgue points.

\section{Duality of square functions}

\setcounter{lemma}{-1}

We continue with the notation from Section 4.
To obtain a general form of the square function property (S1) we first
identify the dual of $\gamma(H,X)$. We denote by $\gamma_{+}'(H,X)$ the space of
all bounded operators $v:H \to X$ such that
\begin{equation}
\|v\|_{\gamma'} = \sup \lbrace \hspace*{-0.3em}\text{ trace } \hspace*{-0.3em}(v' \circ u): u\hspace*{-0.2em}:\hspace*{-0.2em} H \to X', \|u\|_\gamma \leq 1, 
\mbox{ dim } u(H) < \infty \rbrace \hspace*{-0.2em} < \hspace*{-0.2em} \infty.
\end{equation}
By $\gamma'(H,X)$ we denote the closure of the finite dimensional operators
in $\gamma_{+}'(H,X)$.

\begin{remark}
If $X=Y'$, then in (1) one should take the supremum over all finite dimensional $u: H \to Y'' = X'$. However, by Goldstine's theorem, such $u$ can be approximated by finite dimensional $v:H\to Y$ in the $\sigma(Y'',Y')$ topology and it is enough to consider such $v:H\to Y$ in (1).
\end{remark}

\begin{proposition}
The dual of $\gamma(H,X)$ with respect to trace duality is $\gamma_{+}'(H,X')$.
If $u \in \gamma(H,X)$, $v \in \gamma_{+}'(H,X')$, then $v'u$ is in the trace class
$\calS_1(H)$ and for $<u,v>= \operatorname{trace} (v' \circ u)$
$$
|<u,v>| \leq \|v'u\|_{\calS_1(H)} \leq \|u\|_\gamma \cdot \|v\|_{\gamma'}.
$$
Furthermore, if $H$ is separable with an orthonormal basis $(e_j)$ then
$$
<u,v> = \sum_j <u e_j, v e_j>_X.
$$
\end{proposition}

\begin{proof}
If $\phi \in \gamma(H,X)'$, then $(h,x) \in H \times X \to \phi(x \otimes h)$
is a bounded bilinear form, which defines a bounded operator $v \in 
B(H,X')$ with $\phi(x \otimes h)=$ $<vh,x>$. For $u = \displaystyle
\sum_{i=1}^n x_i \otimes e_i \in \gamma(H,X)$, where $(e_i)$ is an orthonormal
sequence, we have
$$
\phi(u) = \sum_{i=1}^n <x_i,ve_i> = \operatorname{trace} \biggl( 
\sum_{i=1}^n v' x_i \otimes e_i \biggr) = \operatorname{trace} (v'u).
$$
By Remark 5.0., the first claim follows.
Now we estimate the trace class norm of $v'u$ for the finite dimensional
$u$ above. To this end we choose a second orthonormal sequence $f_j,
\varepsilon_j \in \C$ with $| \varepsilon_j| = 1$ and a unitary operator
$J$ on $H$ such that $J e_j = \varepsilon_j f_j$ and for all $n \in \N$
\begin{equation*}
\begin{split}
 \sum_{j=1}^n |< e_j, v'u(f_j)>| &= \sum_{j=1}^n <e_j,v'u(\varepsilon_jf_j)> = \sum_j < ve_j,(u \cdot J)e_j > \\
&= \operatorname{trace} (v'u \cdot J) \leq \|v\|_{\gamma'} \cdot \|u \cdot J\|_{\gamma} \leq \|J\| \cdot \|v\|_{\gamma'}\|u\|_\gamma
\end{split}
\end{equation*}
by 4.3 applied to $J$. Now \cite{DJT}, Theorem 4.6, implies that 
$$
|\operatorname{trace}(v'u)| \leq \|v'u\|_{\calS_1(H)} \leq \|v\|_{\gamma'}
\|u\|_\gamma.
$$
We have seen also that
$$
\operatorname{trace} (v'u) = \sum_{i=1}^n <x_i,v_ie_i> = \sum_{i=1}^n
<ue_i,v_ie_i>.
$$
The general case follows now by an approximation argument, since finite
dimensional operators of the form of $u$ are dense in $\gamma(H,X)$.
\end{proof}

\begin{proposition}
a) For all Banach spaces $X$ and $v \in \gamma_{+}(H,X')$ we have $\|v\|_{\gamma'}
\leq \|v\|_\gamma$ and in particular for $u \in \gamma(H,X)$
\[
| \operatorname{trace}(v'u)| \leq \|u\|_{\gamma(X)} \|v\|_{\gamma'(X')}.
\]

b) If $X$ has type larger than $1$ then $\gamma(H,X')$ and $\gamma'(H,X')$ are
isomorphic.
\end{proposition}

\begin{proof}
a) Let $e_j$ be an orthonormal system in $H$.
Then with 5.1 and H\"older's inequality 
\begin{equation*}
\begin{split}
\operatorname{trace}(v'u) &= \sum_j <u(e_j),v(e_j)> \\
& = \E < \sum_i g_iu(e_i), \sum g_j v(e_j)> \leq \|u\|_\gamma \cdot \|v\|_\gamma.
\end{split}
\end{equation*}
Hence $\|v\|_{\gamma'} \leq \|v\|_\gamma$ by the definition of $\|\cdot\|_{\gamma'}$.

b) is shown in \cite{Pi5}.
\end{proof}

\begin{corollary}
Let $S \in B(H_1,H_2)$ and $\Skr$ be its extension $\Skr:
\gamma(H_1,X) \to \gamma(H_2,X)$ as in 4.8. Then the dual map $(\Skr)' : 
\gamma_{+}'(H_2,X') \to \gamma_{+}'(H_1,X')$ with respect to trace duality
is given by $(\Skr)'(v) = v \circ S$ for $v \in \gamma(H_2,X')$ and maps
$\gamma(H_2,X')$ into $\gamma(H_1,X')$.
\end{corollary}

\begin{proof}
For all $u \in \gamma(H_1,X)$ we have with \cite{DJT}, Lemma 6.1, 
\begin{equation*}
\begin{split}
< (\Skr)'(v),u> & = <v, \Skr u> = \operatorname{trace} (v'u \circ S') \\
& = \operatorname{trace} (S' \circ v' \circ u) = \operatorname{trace} ((vS)' \circ u) =
<vS,u>.
\end{split}
\end{equation*}
It is clear that $v \mapsto v \circ S$ maps finite dimensional operators
into finite--dimensional operators.
\end{proof}

\begin{remark}
Let $T \in B(H_1,H_2)$. If we apply this procedure to $S = T' \in 
B(H_2,H_1)$, then we can extend $T: H_1 \to H_2$ to an operator 
$\Tkr:\gamma'(H_1,X') \to \gamma'(H_2,X')$ by $\Tkr u = u \circ T'$ such that

a) For the extension $\Skr: \gamma(H_2,X) \to \gamma(H_1,X)$ with $S=T'$, we have $(\Skr)' = \Tkr$. 

b) $< x,\Tkr(u)>_X = T(<x,u>_X)$ for all $u \in \gamma'(H_1,X')$ and
$x \in X$.

c) If $f \in \mathcal{P}_2(\Omega_1,X')$ and $g \in \mathcal{P}_2(\Omega_2,X')$
are such that $u_f \in \gamma'(L_2(\Omega_1),X')$ and $u_g = \Tkr u_f$,
then
$$
\|u_g\|_{\gamma'} \leq \|T\| \|u_f\|_{\gamma'}
$$
and
$$
<x,g>_X = T(<x,f>_X) \quad \mbox{ for } x \in X.
$$
\end{remark}
For a function $f \in \mathcal{P}_2(\Omega,\mu,X)$ on a $\sigma$--finite
measure space we use the notation
$$
\|f\|_{\gamma'(\Omega,X)} = \|u_f\|_{\gamma'(L_2(\Omega),X)}
$$
and $\gamma_{+}'(\Omega,X) (\gamma'(\Omega,X))$ denotes the space of functions
for which \allowbreak $u_f \in \gamma_{+}'(\Omega,X)$ $(u_f \in \gamma'(\Omega,X))$ with this
norm. Note that the convergence results in 4.10 also apply to the 
$\gamma'(H,X)$--norm (with the same justification).

In particular, Examples 4.9 can be adopted to the $\gamma'(H,X')$--norm.
(All of this can be justified with the same calculation as in 4.4, 4.8, and
4.9.)

Now 5.1 and 5.2 take a form that corresponds to the square function
property (S1).

\begin{corollary}
If $f \in \gamma(\Omega,X)$ and $g \in \gamma_{+}'(\Omega,X')$ then 
$$
\operatorname{trace}(u_f \circ u_g') = \int_{\Omega} < f(\omega), g(\omega)> d\mu
(\omega)
$$
and
\begin{equation*}
\begin{split}
\int_{\Omega} |< f(\omega),g(\omega)>| d\mu(\omega) & \leq \|f\|_{\gamma(\Omega,X)} \cdot
\|g\|_{\gamma'(\Omega,X')} \\
& \leq \|f\|_{\gamma(\Omega,X)} \cdot \|g\|_{\gamma(\Omega,X')}.
\end{split}
\end{equation*}
If $X$ has type larger than $1$ then $\|f\|_{\gamma(\Omega,X')}$ and
$\|f\|_{\gamma'(\Omega,X')}$ are equivalent.
\end{corollary}

\begin{proof}
Let $E_k$, $k = 1,\dots,n$, be a partition of $\Omega$ with $0 < \mu(E_k)
< \infty$ and consider the averaging projection
$$
Ph = \sum_{k=1}^n \mu(E_k)^{-1} \int\limits_{E_k} f(\omega)d\mu(\omega)
\chi_{E_k}.
$$
We write $Pf = \displaystyle \sum_{k=1}^n x_k \chi_{E_k}, \; Pg = \sum_{k=1}^n x_k' \chi_{E_k}$ for some
$x_k \in X, x_k' \in X'$ and $u_{Pf} = \displaystyle \sum_{k=1}^n \mu (E_k)^{1/2} x_k
\otimes l_k,\; u_{Pg} = \displaystyle \sum_{k=1}^n \mu(E_k)^{1/2} x_k' \otimes l_k$ with
$l_k = \mu(E_k)^{- 1/2} \chi_{E_k}$. Then
\begin{equation}
\begin{split}
& \int_{\Omega} < Pf(\omega),g(\omega)>d\mu (\omega) = \int_{\Omega} < Pf(\omega), 
Pg(\omega) >d\mu(\omega) \\
& = \sum_{k=1}^n \mu(E_k) < x_k,x_k'> = \sum_{k=1}^n <u_{Pf}l_k,u_{Pg}l_k> \\
& = \operatorname{trace}(u_{Pg}'u_{Pf}) = \operatorname{trace}((u_gP)'(u_fP)) \\
& = \operatorname{trace}(u_g'u_fPP') = \operatorname{trace}(u_g'u_{Pf}) \leq
\|g\|_{\gamma'(\Omega,X')} \|Pf\|_{\gamma(\Omega,X)} \\
& \leq \|g\|_{\gamma'(\Omega,X')} \cdot \|Pf\|_{\gamma(\Omega,X)} \leq \|g\|_{\gamma'(\Omega,X')} \|f\|_{\gamma(\Omega,X)}
\end{split}
\end{equation}
where we used Lemma 6.1 of \cite{DJT}, Proposition 5.2 and 4.3. Since $f \in \gamma(\Omega,X)$ we can find a sequence of 
projections $P_k$ with $\|f-P_kf\|_{\gamma(\Omega,X)} \to 0$ as $k\to\infty$. For a fixed $k$ and $l$
choose $m \in L_{\infty}(\Omega)$ with $|m(\omega)| = 1$ for almost all $\omega\in\Omega$ so that 
$$
|<P_kf(\omega),g(\omega)> - <P_lf(\omega),g(\omega)>| = <(P_k-P_l)f(\omega),
m(\omega) g(\omega)>.
$$
Then by (2)
\begin{equation*}
\begin{split}
& \int_{\Omega} |<P_kf(\omega),g(\omega)> - <P_lf(\omega),g(\omega)>|d\mu(\omega) \\
& \leq \|mg\|_{\gamma'(\Omega,X')} \|P_kf - P_lf\|_{\gamma(\Omega,X)} \leq \|g\|_{\gamma'(\Omega,X')}
\|P_kf-P_lf\|_{\gamma(\Omega,X)} \to 0
\end{split}
\end{equation*}
for $k,l\to\infty$ and by Proposition 5.2a)
$$
| \operatorname{trace}(u_g' \cdot u_{P_nf}) - \operatorname{trace}(u_g' \cdot
u_f) | \leq \|u_g\|_{\gamma'(\Omega,X')} \|u_{P_nf} - u_f\|_{\gamma(\Omega,X)} \to 0
$$
for $k,l \to \infty$. Hence, $\omega \mapsto <f(\omega),g(\omega)>$ is integrable and with (2)
$$
\int_{\Omega} <f(\omega),g(\omega)>d\mu(\omega) = \lim_{n\to\infty} \operatorname{trace}(u_g' u_{P_nf}) =
\operatorname{trace}(u_g' u_f).
$$
Hence the required inequality holds too.
\end{proof}

We dualize now 4.11 and assume again that $\Omega$ is a locally compact
metric space without isolated points and a positive Borel measure $\mu$.

\begin{corollary}
Let $N:\Omega \to B(X,Y)$ be a strongly continuous map. Suppose $\tau =
\{ N(\omega):\omega \in \Omega \}$ is $\gamma$--bounded with constant $K$.
Then for all $g \in \gamma_{+}'(\Omega,Y')$ 
$$
\| N(\cdot)'[g(\cdot)] \|_{\gamma'(\Omega,X')} \leq K \|g\|_{\gamma'(\Omega,Y')}.
$$
In particular, if $X$ and $Y$ have type larger than $1$ then $\{ N'(\omega):
\omega \in \Omega \}$ is $\gamma$--bounded in $B(Y',X')$.
\end{corollary}

\begin{proof}
For $f \in \gamma(\Omega,X)$ with $\|f\|_{\gamma(\Omega,X)} \leq 1$ we have
\begin{equation*}
\begin{split}
\int_{\Omega} < f(\omega), N(\omega)' g(\omega) >_X &d\mu (\omega) =
\int_{\Omega} < N(\omega)f(\omega), g'(\omega)>_Y d\mu (\omega) \\
& \leq \|g\|_{\gamma'(\Omega,Y')} \| N(\cdot)f(\cdot)\|_{\gamma(\Omega,Y)} \leq K \|g\|_{\gamma'(\Omega,Y')}
\end{split}
\end{equation*}
by 4.11. Now take the supremum over $f$ to obtain the norm in $\gamma'(\Omega,X')=\gamma(\Omega,X)'$. The last statement follows from 5.2b).
\end{proof}

We record an immediate consequence

\begin{example}
Let $t \in I \to N(t) \in B(X)$ be strongly continuous on an interval $I$
and $\gamma$--bounded with constant $C$. Then for $h \in L_2(I)$ we have
\begin{eqnarray*}
\|h(\cdot)N(\cdot)x\|_{\gamma(I,X)} & \leq & C \|h\|_{L_2(I)} \| x \|, \quad
x \in X \\
\|h(\cdot)N'(\cdot)x'\|_{\gamma'(I,X')} & \leq & C \|h\|_{L_2(I)} \|x'\|, \quad
x' \in X'.
\end{eqnarray*}
This example is one motivation for us to collect some criteria for
$\gamma$-- \linebreak boundedness.
\end{example}

First we recall a well--known convexity result.
Corollary 5.6 will often be used in connection with the following criteria
for $\gamma$--boundedness.

\begin{lemma}
Let $t \in \Omega \to N(t) \in B(X,Y)$ be strongly measurable and suppose
that $\{N(t):t \in \Omega\}$ is $R$--bounded ($\gamma$--bounded) with
$R$--bound ($\gamma$--bound) $C$. For $h \in L_1(\Omega,\mu)$ define 
\begin{equation}
N_h(x) = \int\limits_{\Omega}h(t)N(t) x d\mu(t), \quad x \in X.
\end{equation}
Then the set $\{N_h:\|h\|_{L_1(\Omega)} \leq 1 \}$ is $R$--bounded 
($\gamma$--bounded) in $B(X,Y)$ and with bound $2C$.
\end{lemma}

\begin{proof}
We have $N_h \in \tau := \overline{\mbox{absco}} \{N(t) : t \in \Omega \}$ for all
$h$, $\|h\|_{L_1(\Omega)} \leq 1$, where the closure is with respect to the strong
operator topology. Now use the $R$--boundedness of $\tau$, which is shown in \cite{CPSW}. The argument for $\gamma$--boundedness is similar.
\end{proof}

\begin{lemma}
Let $t \in \Omega \to N(t) \in B(X,Y)$ be strongly integrable and suppose
that there is a constant $C$ with
$$
\int\limits_{\Omega}\|N(t)x\| d\mu(t) \leq C\|x\|, \quad x \in X.
$$
Then the set $\{N_h:\|h\|_{L_{\infty}(\Omega)} \leq 1\}$ with $N_h$ as in
(3) is $R$--bounded and $\gamma$--bounded with bound $2C$.
\end{lemma}

\begin{proof}
For $x_1,\dots,x_m \in X$ and $h_1,\dots,h_m \in L_{\infty}(\Omega)$ with
$\|h_j\|_{L_{\infty}} \leq 1$ we obtain from Kahane's inequality and
Fubini's theorem
\begin{equation*}
\begin{split}
& \E \biggl| \biggl| \sum_k g_k N_{h_k} x_k \biggr| \biggr| \leq
\int\limits_{\Omega}\E \biggl| \biggl| \sum_k g_k h_k(t)N(t)x_k \biggr| \biggr|
d\mu(t) \\
& \leq 2 \int\limits_{\Omega} \E \biggl| \biggl| N(t) \biggl[ \sum_k g_k x_k \biggr]
\biggr| \biggr| d\mu(t) \leq C \E \biggl| \biggl| \sum_k g_k x_k \biggr| \biggr|.
\end{split}
\end{equation*}
Of course the same reasoning works for Rademacher functions.
\end{proof}

At one point we will need a generalization of this observation to vector
measures (see \cite{DU} for definitions).
Let $\Sigma$ be a $\sigma$--Algebra on $\Omega$ and denote by ${\cal L}^\infty$ the space of $\Sigma$-measurable, bounded functions, i.e., the space of uniform limits of $\Sigma$--measurable step functions. We endow ${\cal L}^\infty$ with the supremum norm.

\begin{lemma}
Suppose that $U \in \Sigma \to P(U) \in B(X,Y)$ is a vector measure, such
that for all $x \in X$ the vector measure $U \in \Sigma \to P(U) x$ is 
$\sigma$--additive. Then the set 
$$
\{ T_h: h\in{\cal L}^\infty, \|h\|_\infty \leq 1 \}, \quad T_hx = \int_\Omega h(\omega)dP(\omega)x
$$
is $R$--bounded and $\gamma$--bounded.
\end{lemma}

\begin{proof}
By the uniform boundedness principle there is a $C < \infty$ with \linebreak
$\operatorname{Var}(P(\cdot)x) \leq C\|x\|$. 
First we consider step functions $h_k = \displaystyle
\sum_{j} a_{k,j} \chi_{A_j}$ where the $A_j \in \Sigma$
form a measurable partion of $\Omega$ and $|a_{k,j}| \leq 1$. Then again by
Kahane's inequality
\begin{equation*}
\begin{split}
 \E \biggl| \biggl| \sum_k g_k T_{h_k} x_k \biggr| \biggr| &\leq \sum_j 
\E \biggl| \biggl| \sum_k g_k a_{k,j} P(A_j)x_k \biggr| \biggr| \leq 2 \sum_j \E \biggl| \biggl| \sum_k g_k P(A_j)(x_k) \biggr| \biggr| \\
&= 2 \E \sum_j \biggl| \biggl| P(A_j) \biggl[ \sum_k g_k x_k \biggr] 
\biggr| \biggr| \leq 2 C \E \biggl| \biggl| \sum_k g_k x_k \biggr| \biggr|.
\end{split}
\end{equation*}
Since we can approximate functions in ${\cal L}^\infty$ by step functions and the
closure of a $R$--bounded ($\gamma$--bounded) set in the strong operator topology is
again $R$--bounded ($\gamma$--bounded), the claim follows.
\end{proof}

The following fact will be important in estimating square functions.

\begin{lemma}
Let $X$ have finite cotype. Then there is a constant $C$ so that
\begin{eqnarray*}
\|(x_i)_i\|_{\gamma(\N,X)} & \leq & C \sup_{\delta_i = \pm 1} 
\biggl| \biggl| \sum_i \delta_i x_i \biggr| \biggr|_X \\
\|(x_i')_i\|_{\gamma(\N,X')} & \leq & C \sup_{\delta_i = \pm 1}
\biggl| \biggl| \sum_i \delta_i x_i' \biggr| \biggr|_{X'}.
\end{eqnarray*}
\end{lemma}

\begin{proof}
Since $X$ has finite cotype we have
$$
\| (x_i)_i\|_{\gamma(\N,X)} \leq C\E \biggl| \biggl| \sum_i r_i x_i 
\biggr| \biggr| \leq \sup_{\delta_i = \pm 1} 
\biggl| \biggl| \sum_i \delta_i x_i \biggr| \biggr|.
$$
For the second statement observe that for $x_1,\dots,x_m \in X$, $x'_1,\ldots,x'_n \in X'$
\begin{equation*}
\begin{split}
& \biggl| \sum_{i=1}^n < x_i, x_i' > \biggr| = \biggl| \E 
< \sum_i r_i x_i, \sum_j r_j x_j' > \biggr| \\
& \leq \biggl( \E \biggl| \biggl| \sum r_i x_i \biggr| \biggr|^2 
\biggr)^{1/2} \biggl( \E \biggl| \biggl| \sum_j r_j x_j'
\biggr| \biggr|^2 \biggr)^{1/2} \\
& \leq C \| ( x_i)_i \|_{\gamma(\N,X)} \sup_{\delta_i = \pm 1}
\biggl| \biggl| \sum_j \delta_j x_j' \biggr| \biggr|.
\end{split}
\end{equation*}
Now take the supremum over $(x_i)_i$ with $\| (x_i)_i \|_{\gamma(\N,X)} \leq 1$.
\end{proof}

\begin{remark}
Before we apply the square functions $\| \cdot \|_\gamma$ and $\| \cdot \|_{\gamma'}$ to the $H^\infty$ calculus, we summarize our extensions of the basic properties (S1), (S2) and (S3) from Section 2 and 3 to the Banach space setting: Let $X$ be a Banach space and $(\Omega,\mu)$ a $\sigma$-finite measure space.
\begin{itemize}
\item[(S1)] (H\"older inequality, see 5.5) If $f\in\gamma(\Omega,X)$ and $g\in\gamma(\Omega,X')$, or even $g\in\gamma_+(\Omega,X')$, then
\[ \int_\Omega |<f(\omega),g(\omega)>_X| d\mu(\omega) \leq \| f \|_{\gamma(\Omega,X)} \cdot \| g\|_{\gamma'(\Omega,X')}. \]
\item[(S2)] (Pointwise Multiplier, see 4.11, 5.6) Let $N:\Omega \to B(X)$ be a strongly continuous map with a $\gamma$--bounded range $\tau := \lbrace N(\omega): \omega \in \Omega \rbrace$. For $f\in\gamma(\Omega,X)$ and $g\in\gamma'(\Omega,X')$, we have
\begin{align*}
\| N(\cdot)[f(\cdot)] \|_{\gamma(\Omega,X)} &\leq \gamma(\tau) \| f\|_{\gamma(\Omega,X)} \\
\| N(\cdot)'[g(\cdot)] \|_{\gamma'(\Omega,X')} &\leq \gamma(\tau) \| g\|_{\gamma'(\Omega,X')}
\end{align*}
\item[(S3)] (Extension property, see 4.9, 5.4) Let
\begin{equation*}
Kf(\omega_2) = \int_{\Omega_1} k(\omega_2,\omega_1)f(\omega_1) d\mu_1(\omega_1)
\end{equation*}
be a kernel operator that is defined on a dense subset of $L_2(\Omega_1,\mu_1)$ and extends by continuity to a bounded operator $K: L_2(\Omega_1,\mu_1) \to L_2(\Omega_2,\mu_2)$ ($K$ could be e.g. the Fourier transform or a singular integral such as the Hilbert transform). Then, applying $K$ formally to $f\in\gamma(\Omega_1,X)$ or $K'$ to $g\in\gamma'(\Omega_2,X')$ defines bounded operators ${\cal K}: \gamma(\Omega_1,X) \to \gamma(\Omega_2,X)$ or ${\cal K}': \gamma'(\Omega_2,X') \to \gamma'(\Omega_1,X')$ with
\begin{align*}
\| {\cal K} \|_{\gamma(\Omega_2,X)} &\leq \| K \|_{L_2 \to L_2} \| f \|_{\gamma(\Omega_1,X)}, \\
\| {\cal K}' \|_{\gamma(\Omega_1,X')} &\leq \| K \|_{L_2 \to L_2} \| g \|_{\gamma'(\Omega_2,X')}.
\end{align*} 
\end{itemize}

\end{remark}

\section{Generators of $C_0$--groups}

In connection with square function estimates we consider operators
$A$ of strip--type on a Banach space $X$ such that $\{ R(\lambda,A):| \mbox{Re} \lambda| >
a \}$ is not only bounded, but even $\gamma$--bounded ($R$--bounded).
We call such operators of $\gamma$--strip--type ($R$--strip--type) and 
$w_\gamma(A)$ (or $w_R(A))$ is the smallest $a$ for which the above $\gamma$--boundedness
conditions hold.

\begin{lemma}
If $A$ generates a $C_0$--group $T_t$, then $A$ is of $\gamma$-- (and $R$--)
strip type with $w_{\gamma}(A), w_R(A) \leq \omega(T_t)$.
\end{lemma}

\begin{proof}
For Re $\lambda \geq a$ we have with $h_{\lambda}(t) = e^{-(\lambda-a)t}$
$$
R(\lambda,A)x= \int\limits_0^{\infty} h_{\lambda}(t) e^{-at} T_txdt.
$$
Since $\|h_{\lambda}\|_{\infty} \leq 1$ and $e^{-at} \|T_t\|$ is integrable
we may apply Lemma 5.10. Similarly for Re $\lambda \leq -a.$
\end{proof}

Our results on the $H^{\infty}(S(a))$--calculus for strip--type operators
are based on the following characterization in terms of square functions.
(These square functions are finite by the same argument we employed in 3.2
and 4.6b).)

\begin{theorem}
Let $A$ be a $\gamma$--strip--type operator on a Banach space $X$. Consider the
conditions 
\begin{itemize}
\item[a)] 
$A$ generates a $C_0$--group $(T_t)_{t \in \R}$ such that for one (all)
$a > \omega(T_t)$ there is a constant $C$ with 
\begin{eqnarray*}
\|e^{-a\cdot}T_{(\cdot)}x\|_{\gamma(\R_{+},X)} & \leq & C \|x\|, \quad x \in \cal{D}(A), \\
\|e^{-a\cdot} T_{(\cdot)}'x' \|_{\gamma'(\R_{+},X)} & \leq & C \|x'\|, \quad x' \in \cal{D}(A').
\end{eqnarray*}

\item[b)] 
For one (all) $a$ with $|a| > w_\gamma(A)$ there is a constant $C$, such that
\begin{eqnarray*}
\|R(a+i \cdot,A)x \|_{\gamma(\R,X)} & \leq & C \|x\|, \quad x \in \cal{D}(A), \\
\|R(a+i \cdot,A)'x'\|_{\gamma'(\R,X')} & \leq & C \|x'\|, \quad x' \in \cal{D}(A').
\end{eqnarray*}

\item[c)] 
For one (all) $a$ with $|a| > w_\gamma(A)$ and $x \in X$ there is a constant
$C$ such that for $x \in \cal{D}(A)$
$$
\frac{1}{C}\|x\| \leq \| R(a+i \cdot,A)x\|_{\gamma(\R_{+},X)} \leq C \|x\|.
$$

\item[d)] 
$A$ has an $H_{\infty}(S(b))$--calculus for one (all) $b > w_\gamma(A)$.
\end{itemize}

Then a) $\Longrightarrow$ b) $\Longrightarrow$ c) $\Longrightarrow$ d) 
always. If $X$ has finite cotype, then d) $\Longrightarrow$ a). Furthermore, in this
case $w_\gamma(A) = \omega(T_t) = w_{H^{\infty}}(A)$.
\end{theorem}

\begin{remark}

a) The proof will show that for all operators of strip--type, a fixed $a$ and $b > a$
we always have
a) $\Longrightarrow$ b) $\Longrightarrow$ c) $\Longrightarrow$ d). 
If $X$ has finite cotype and $b < a$ then d) $\Longrightarrow$ a).

b) Since d) and therefore all conditions imply that $A$ is of $\gamma$--strip--type,
we could have formulated the theorem for operators of strip--type.
The assumption that $A$ is of $\gamma$--strip--type is only used to 
vary the size of the strip.

c) If $X'$ has also finite cotype then we can replace the $\gamma'$--norm in b)
by the $\gamma$--norm, i.~e.~the second condition takes the form
$$
\|R(b+i \cdot, A)'x'\|_{\gamma(\R,X')} \leq C \|x'\|.
$$

d) The proof below will give for b) $\Longrightarrow$ d) the following
estimate:
$$
\| f(A)\| \leq 2aC^2 \| f \|_{H^\infty (S(a))}
$$
where $C$ is the constant in condition 6.2b) (cf Remark 2.3)
\end{remark}

We will need the following lemma: 

\begin{lemma}
Let $X$ have finite cotype. Assume that $A$ has an $H^{\infty}(S(\omega))$--calculus.
Then $A$ generates a group $T_t$ with $\omega(T_t) \leq \omega$ and for
$a>\omega$ there is a constant $C$ such that
\begin{eqnarray*}
\|e^{-a|\cdot|}T_{(\cdot)}x\|_{\gamma(\R,X)} & \leq & C C_{H^{\infty}} e^{\pi(a+\omega)} 
\frac{1}{a-\omega} \|x\|, \quad x \in X \\
\|e^{-a|\cdot|}T_{(\cdot)}'x'\|_{\gamma'(\R,X')} & \leq & C C_{H^{\infty}} e^{\pi(a+\omega)}
\frac{1}{a-\omega} \|x'\|, \quad x' \in X'
\end{eqnarray*}
where $C$ depends only on $X$, and $C_{H^{\infty}}$ is the bound for the
$H^{\infty}(S(\omega))$--calculus.
\end{lemma}

\begin{proof}
Choose a nonnegative function $h \in C^{\infty}(\R)$ with support in
$(-\pi, \pi)$ and $\displaystyle \sum_{k\in\Z} h(t-k) = 1$ on $\R$. Put $h_k(t) =
h(t-k)$ for $k\in\Z$ and $g(t) = [\cosh (at)]^{-1}$. Then 
\begin{equation*}
\begin{split}
& \|e^{-a|\cdot|}T_{(\cdot)} x\|_{\gamma(\R,X)} \leq C \|g(\cdot)T_{(\cdot)}x\|_{\gamma(\R,X)} \leq C
\sum_{k=-\infty}^{\infty} \|h_k(\cdot)g(\cdot)T_{(\cdot)}x\|_{\gamma(\R,X)} \\
& = C \sum_{k=-\infty}^{\infty} \| {\cal F}[h(\cdot)g(\cdot+k)T_{\cdot +k}x]\|_{\gamma(\Z,X)}
\end{split}
\end{equation*}
where $\Fkr:\gamma(L_2(-\pi, \pi),X) \to \gamma(l_2(\Z,X))$ is the extension in the
sense of 4.9 of the discrete Fouriertransform ${\cal F}$, i.~e.~
\begin{equation*}
\begin{split}
{\cal F}[h(\cdot)g(\cdot +k)T_{\cdot +k}x](n) & = \frac{1}{\sqrt{2 \pi}}
\int\limits_{- \pi}^{\pi} e^{int}h(t)g(t+k)T_{t+k}xdt \\
& = b_{k,n}(A)x
\end{split}
\end{equation*}
with $b_{k,n}(\lambda) = \frac{1}{\sqrt{2 \pi}} \int\limits_{- \pi}^{\pi}
e^{int}h(t)g(t+k)e^{\lambda(t+k)}dt$ for $n\in\Z$.
We check now the estimate
\begin{equation}
\sum_{k=- \infty}^{\infty} \sup_{\lambda \in S_{\omega}} \sum_{n=-\infty}^\infty |b_{k,n}
(\lambda)| \leq \frac{C}{a- \omega} e^{\pi(a+ \omega)}.
\end{equation}
For all $n,k\in\Z$ and $\lambda \in S_{\omega}$
$$
|b_{k,n}(\lambda)| \leq C \|h_{\infty} \| e^{\pi(a+\omega)} e^{(\omega - a)|k|}
$$
and for $\lambda \in S_{\omega}$ with $\lambda \not= in$ and $k \not= 0$ by partial        
integration 
\begin{equation*}
\begin{split}
& |b_{k,n}(\lambda)| \leq \biggl|  \frac{1}{(in + \lambda)^2} \int\limits_{-\pi}^{\pi} 
e^{(in+ \lambda)t} \frac{d^2}{dt^2} [h(t)g(t+k) e^{-a\lambda k}]dt \biggr|  \\
& \leq C \frac{e^{(\omega-a)|k|}}{|in+ \lambda|^2} \int\limits_{-\pi}^{\pi}
(|g''(t)|+|g'(t)|+|g(t)|) dt \\
& \leq C_1 \frac{e^{(\omega-a)|k|}}{|in + \lambda|^2}(1+a+a^2)e^{\pi(\omega+a)}.
\end{split}
\end{equation*}
Combining these two estimates we obtain for every $k$
\begin{equation*}
\begin{split}
& \sup_{\lambda \in S_{\omega}} \sum_n |b_{k,n}(\lambda)| \leq \sup_{\lambda
\in S_{\omega}} \biggl( \sum_{|n-\lambda|<1} |b_{k,n}(\lambda)| + 
\sum_{|n-\lambda|\geq 1} |b_{k,n}(\lambda)| \biggr) \\
& \leq C e^{(\omega-a)|k|}e^{\pi(a+\omega)} + C_1 \biggl(\sum_{n>1} 
\frac{1}{n^2} \biggr) e^{(\omega-a)|k|} e^{\pi(\omega+a)}
\end{split}
\end{equation*}
and summing over $k$ we have established (1).

Since $X$ has finite cotype we can continue our estimation of $\| e^{-|a|t}T_t x \|_\gamma$ with Lemma 5.11 and use boundedness of the $H^\infty$-calculus to obtain
\begin{equation*}
\begin{split}
& \sum_{k=-\infty}^{\infty} \| (b_{k,n}(A)x)_n \|_{\gamma(\Z,X)} \leq
\sum_{k=-\infty}^{\infty} \sup_{|\delta_n|=1} \biggl| \biggl| \sum_n
\delta_n b_{k,n}(A)x \biggr| \biggr| \\
& \leq C_{H^{\infty}} \sum_{k=-\infty}^{\infty} \sup_{|\delta_n|=1}
\sup_{\lambda \in S(\omega)} \biggl| \sum_n \delta_n b_{k,n}(\lambda) \biggr|
\|x\| \leq C C_{\infty} \frac{1}{a-\omega} e^{\pi(a+\omega)}.
\end{split}
\end{equation*}
The second claim can be shown in the same way since $\|f(A)^{\ast}\|\leq$ \linebreak
$C_{H^{\infty}} \|f\|_{H^{\infty}(S(\omega))}$ and using now the estimate for $\gamma'(\Z,X')$ in Lemma 5.11.
\end{proof}

\begin{proof}[Proof of 6.2]
We can repeat the proof of 2.2, replacing the norm $\|f(\cdot)\|_{L_2(I,H)}$
by $\|f(\cdot)\|_{\gamma(I,H)}$ and (S1), (S2) and (S3) by 5.5, 4.11 and 4.8. Since in 4.11
we need $\gamma$--boundedness of $N(\cdot)$ instead of mere boundedness, we have to
appeal to Lemma 6.1~in addition. The simple fact that $e^{-at}T_tx$
belongs to $L_2(\R_{+},X)$ for $a>\omega(T_t)$ has to be replaced by
Lemma 6.4. 
\end{proof}

The argument of 2.2 can also be used to extend an $H^{\infty}$--calculus
to an operator--valued functional calculus: Let $A$ be an operator of strip--type on
$X$ and denote by ${\cal A}$ the algebra of all operators in $B(X)$ that
commute with the resolvent of $A$. Then $RH^{\infty}(S(a),{\cal A})$ is the
space of bounded analytic functions $F:S(a) \to {\cal A}$, such that the range
$\{ F(\lambda):\lambda \in S(a)\}$ is $R$--bounded. For $F \in RH^{\infty}
(S(a),{\cal A})$ and $\varphi \in H_0^\infty(S(a))$, fix $b$ with $w(A) < b < a$ and let
$$
(\varphi F)(A) = \frac{1}{2\pi i} \int\limits_{\partial S(b)} \varphi(\lambda)
F(\lambda)R(\lambda,A)d\lambda.
$$
We say that $A$ has an $RH^{\infty}(S(a),{\cal A})$--calculus if there is
a constant $C$ such that for all $\varphi \in H_0^{\infty}(S(a))$ we have 
$\| \varphi F(A)\| \leq C \| \varphi F\|_{H^{\infty}(S(a))}$. In this case
we can define a bounded operator $F(A)x:= \displaystyle \lim_n \varphi_n
F(A)x, x \in X$, where $\varphi_n$ is a sequence in $H_0^{\infty}(S(a))$ with
$|\varphi_n(\lambda)| \leq 1$ and $\varphi_n(\lambda) \to 1$ for $n \to \infty$
and $\lambda \in S(a)$.

The next result holds under weaker assumptions on $X$, see \cite{KW2}, but
in the situation considered here we have a simple proof.

\begin{corollary}
Let $X$ have finite cotype. If $A$ has an $H^{\infty}(S(a))$--calculus for some
$a > w(A)$, then $A$ has an $RH^{\infty}(S(b),{\cal A})$--calculus for all
$b > a$.
\end{corollary}

\begin{proof}
Consider $f(\lambda) = \varphi(\lambda)F(\lambda)$ with $\varphi \in
H^{\infty}(S(b))$ and $F \in RH^{\infty}$ $(S(b),{\cal A})$. 
Using the argument of 2.2c) $\Longrightarrow$ d) and 6.2b) in place of 2.2b) we arrive
at 
\begin{equation*}
\begin{split}
\|f(A)x\| & \leq C \|f(\cdot)R(\cdot,A)x\|_{\gamma(\partial S(a),X)} + \| {\cal K}[f(\cdot)R(\cdot, A)]x\|_{\gamma(\partial S(a),X)}.
\end{split}
\end{equation*}
By the $\gamma$--boundedness of $\{f(\lambda):\lambda \in \partial S(a) \}$, 4.11 and
4.8 applied to ${\cal K}$ we get with 6.2c) 
$$
\|f(A)x\| \leq C_1\|R(\cdot , A)x\|_{\gamma(\partial S(a),X)} \leq C_2\|x\|.
$$
\end{proof}

As a consequence we obtain an $\gamma$--boundedness criterion:

\begin{corollary}
Let $X$ have property $(\alpha)$. If $A$ has an $H^{\infty}(S(a))$--calculus,
then the set $\{f(A):\|f\|_{H^{\infty}(S(a))}  \leqslant 1 \}$ is $\gamma$--bounded.
\end{corollary}

\begin{proof}
Let $\|f_n\|_{H^{\infty}(S(a))} \leq 1$. The $\gamma$--boundedness of the
sequence $f_n(A)$ is equivalent to the boundedness of the operator 
$T((x_n)_n) = (f_n(A)x_n)_n$ on $\gamma(\N,X)$. Put $F(\lambda) = (f_n(\lambda)I_X)$
for $\lambda \in S(a)$. By Kahane's contraction principle $\lambda \in S(a)
\to F(\lambda) \in B(\gamma(\N,X))$ is an analytic function with 
$\gamma$--bounded range in $B(\gamma(\N,X))$. Since $X$ has property
$(\alpha)$, it has also finite cotype and we can check the $\gamma$--boundedness using two independent Rademacher sequences $(r_n)_n$, $(r'_n)_n$ instead of Gaussian random variables and $\gamma(\N,X) \cong \Rad X$. Indeed, for $\lambda_j \in S(a)$ and $\tilde{x}_j = (x_{jn})_n \in
\gamma(\N,X)$
\begin{equation*}
\begin{split}
\E \biggl| \biggl| \sum_j r_j F(\lambda_j) \tilde{x}_j \biggr| 
\biggr|_{\Rad X}^2 &= \E \E \biggl| \biggl| \sum_{j,n} r_j r_n'
f_n(\lambda_j)x_{jn} \biggr| \biggr|_X^2 \\
& \leq C \E \E \biggl| \biggl| \sum_{j,n} r_j r_n' x_{jn} 
\biggr| \biggr|_X^2 = C \E \biggl| \biggl| \sum_j r_j \tilde{x}_j 
\biggr| \biggr|_{\Rad X}^2.
\end{split}
\end{equation*}
Define ${\cal D}(\tilde{A}) = \{ (x_n)_n \in \Rad X : x_n \in 
{\cal D}(A)\}$ and $\tilde{A}((x_n)_n) = (Ax_n)_n$ for $(x_n)_n \in {\cal D}
(\tilde{A})$. Then $\tilde{A}$ has an $H^{\infty}(S(a))$--calculus on
$\gamma(\N,X)$ defined by $f(\tilde{A})(x) = f(\tilde{A})((x_n)_n) =(f(A)x_n)_n$ for $x \in \gamma(\N,X)$ and by 6.5 $\tilde{A}$ has an 
$RH^{\infty}(S(a),{\cal A})$--calculus on $\gamma(\N,X)$. Since $F \in RH^{\infty}
(S(a),{\cal A})$ and for $\varphi \in H_0^{\infty}(S(a))$
$$
\int\limits_{\partial S(a)} \varphi(\lambda)F(\lambda)R(\lambda,\tilde{A})
d\lambda = \biggl( \int\limits_{\partial S(a)} \varphi(\lambda) f_n(\lambda)
R(\lambda,A)d\lambda \biggr)_n
$$
we conclude that $S = F(\tilde{A})$ is bounded on $\gamma(\N,X)$.
\end{proof}

With these tools we can define a joint functional calculus for two
commuting operators $A,B$ of strip--type on $X$. For $f \in H^{\infty}(S(a) \times
S(b))$ and $\varphi \in H_0^{\infty}(S(a)), \psi \in H_0^{\infty}(S(b))$
define
$$
(\varphi \cdot f \cdot \psi)(A,B) = \int\limits_{\partial S(a)}
\int\limits_{\partial S(b)} \varphi(\lambda) f(\lambda, \mu)\psi(\mu)
R(\lambda,A)R(\mu,B)d\lambda d\mu.
$$
We say that $A$ and $B$ have a joint $H^{\infty}$--calculus, if there is a
constant $C$ so that
$$
\|(\varphi f \psi )(A)\| \leq C \| \varphi f \psi \|_{H^{\infty}(S(a)
\times S(b))}
$$
for all $\varphi \in H^\infty_0(S(a))$, $\psi \in H^\infty_0(S(b))$ and $f \in H^{\infty}(S(a) \times
S(b))$. In this case we can define a bounded operator $f(A,B)$ on $X$ by $f(A,B)x =
\displaystyle \lim_n (\varphi_n f \psi_n)(A,B)x$, where $\varphi_n \in H_0^{\infty}
(S(a)), \psi_n \in H_0^{\infty}(S(b))$ are sequences with $|\varphi_n(\lambda)
| \leq 1, \varphi_n(\lambda) \to 1$ for $\lambda \in S(a)$ and $| \psi_n
(\lambda)| \leq 1$ and $\psi_n(\lambda) \to 1$ for $\lambda \in S(b)$ as $n\to\infty$.

\begin{corollary}
Let $X$ have property $(\alpha)$. If $A$ and $B$ are resolvent commuting
operators on $X$ with an $H^{\infty}(S(a))$ and an $H^{\infty}(S(b))$--calculus,
respectively, then $A$ and $B$ have a joint $H^{\infty}(S(a) \times 
S(b))$--calculus. 
\end{corollary}

\begin{proof}
Let $f \in H^{\infty}(S(a) \times S(b)), \varphi \in H_0^{\infty}(S(a))$ and
$\psi \in H_0^{\infty}(S(b))$. Then
$$
(\varphi f \psi)(A,B) = \int\limits_{\partial S(a)} \varphi(\lambda)
F(\lambda) R(\lambda,A)d \lambda
$$
with
$$
F(\lambda) = \int\limits_{\partial S(b)} \psi(\mu)f(\lambda,\mu)R(\mu,B)d\mu.
$$
Corollary 6.6 applied to $B$, implies that $F \in RH(S(a),{\cal A})$. 
Now we may apply 6.5 to $A$ and $F$ and obtain an estimate
\begin{equation*}
\begin{split}
\|(\varphi f \psi)(A,B)\| & \leq C_1 \| \varphi \cdot F\|_{RH^{\infty}(S(b))} \leq C_2 \| \varphi f \psi\|_{H^{\infty}(S(a)\times S(b))}.
\end{split}
\end{equation*}
\end{proof}

We can state now the main result of this section.

\begin{theorem}
Let $A$ be an operator of strip--type on a Banach space $X$.

\noindent a) Suppose that $A$ generates a $C_0$--group $T_t$ on $X$ such 
that for some $a > \omega(T_t)$ the set $\{ e^{-a|t|}T_t:t \in \R\}$ is
$\gamma$--bounded. Then $A$ has an $H^{\infty}(S(b))$--calculus for all $b > a$.

\noindent b) Conversely, assume $X$ has property $(\alpha)$.
If $A$ has an $H^{\infty}(S(a))$--calculus, then $A$ generates a $C_0$--group
$T_t$ such that $\{ e^{-bt}T_t:t \in \R\}$ is $\gamma$--bounded for all $b > a$.
\end{theorem}

\begin{proof}
a) Choose $a_1$ with $a_1 > a$. Since $e^{-(a_1-a)|t|} \in L_2(\R)$ we obtain
from example 5.7 that 
$$
\|e^{-a_1 \cdot}T_{(\cdot)}x\|_{\gamma(\R,X)} \leq C \|x\|, \quad \|e^{-a_1 \cdot}T_{(\cdot)}^{\ast}x^{\ast}
\|_{\gamma(\R,X^{\ast})} \leq C\|x\|. 
$$
By 6.2a) the existence of the $H_{\infty}(S(\omega))$--calculus follows.

b) Note that $e^{-bt}T_t = f_t(A)$ with $f_t(\lambda) = e^{-b|t|}e^{t \lambda},
t > 0$, and $\|f_t\|_{H^{\infty}(S(\omega))} \leq 1$. 
The $\gamma$--boundedness of $\{e^{-b|t|}T_t:t \in \R\}$ follows now from 6.6.
\end{proof}

More can be said about $\gamma$--bounded groups:
Their generators are spectral operators (in the sense of \cite{DS}).
Let $A$ be an operator of strip--type on $X$ with $\sigma(A) \subset i \R$. 
A spectral measure $P$ for $A$ is an additive bounded (projection valued)
function $P: {\cal B} \to B(X)$ on the Borel sets ${\cal B}$ of $i \R$, so that
$$
\Omega \in B \to P(\Omega)x \in X \qquad \mbox{ is } \sigma\mbox{-additive 
for all } x \in X
$$
and
$$
R(\mu,A)x = \int\limits_{i \R} (\mu - \lambda)^{-1} dP(\lambda)x 
\qquad \mbox{ for } \lambda \not\in i \R, x \in X.
$$
By $B_b(i\R)$ we denote the bounded Borel functions on $i\R$.

\begin{corollary}
Assume that $A$ generates a group on $T_t$ on $X$, so that $\{ T_t : t \in 
\R \}$ is $\gamma$--bounded.

\noindent a) Then $A$ has a bounded functional calculus 
$$
\phi_A : C_0(\R) \to B(X).
$$

\noindent b) If $X$ does not contain $c_0$, then there exists a spectral
measure $P$ for $A$ so that for all $f \in B_b(i\R)$
$$
\phi_A(f)x = \int\limits_{i \R} f(\lambda) dP(\lambda)x, \quad x \in X.
$$

\noindent c) Let ${\cal L}$ be the space of bounded functions on $i \R$, which are 
uniform limits of step functions. Then the set $\lbrace \phi_A(f) : f \in B_b(i\R)$,
$\|f\|_{\infty} \leq 1 \}$ is $\gamma$--bounded.

Conversely, if $A$ is a spectral operator, then $\{e^{tA} : t \in \R \}$
is $\gamma$--bounded.
\end{corollary}

\begin{proof}
First we note that $A$ has an $H^{\infty}(S(\omega))$--calculus for all
$\omega > 0$ whose norm only depends on the $\gamma$--boundedness constant $M$ of $\{T_t\}$ and $\{T_t'\}$, more precicely
$$
\|f(A)\| \leq 2\pi M^2 \|f\|_{H^{\infty}(S(\omega))}.
$$
Indeed, by applying the Fourier transform to $e^{-a \cdot}T_{\cdot}$ (cf 4.9) and example 5.7 we have for $a = \frac{\omega}{2}$
\begin{equation*}
\begin{split}
& \|R(a+i \cdot,A)x\|_{\gamma(\R,X)}  \leq \sqrt{2 \pi} \|e^{-a \cdot}T_{(\cdot)} x 
\|_{\gamma(\R_{+},X)} \\
& \leq \sqrt{2 \pi} \biggl( \int\limits_0^{\infty} e^{-2at}dt \biggr)^{1/2}
M\|x\| = \sqrt{\pi} a^{- 1/2} M \|x\|
\end{split}
\end{equation*}
and similarly
$$
\|R(a+ i \cdot,A')x' \|_{\gamma'(\R,X')} \leq \sqrt{\pi} a^{- 1/2} M
\|x'\|.
$$
Now by remark 6.3d)
$$
\|f(T)\| \leq 2\pi M^2 \|f\|_{H^{\infty}(S(\omega))}.
$$
Every $f \in C_0(i \R)$ can be uniformly approximated by a sequence
of functions $g_n$ of the form (see e.~g.~\cite{Co} Corollary 8.3)
$$
g(\lambda) = \sum_{j=1}^n e^{b_j \lambda^2} p_j(\lambda), \quad
b_j > 0, \qquad p_j \mbox{ are polynomials.}
$$
If $\|g_n - g_m \|_{L_{\infty}(i \R)} \leq \varepsilon$ then $\|g_n - g_m
\|_{H^{\infty}(S(\omega))} \leq 2 \varepsilon$ for $\omega > 0$ small enough.
Hence $\|g_n(A) - g_m(A)\| \leq 2\pi M^2 \varepsilon$ and
$(g_n(A))_n$ is a Cauchy sequence in $B(X)$. We put $\phi_A(f) := \displaystyle
\lim_n g_n(A)$.
The fact that $\phi_A$ is well--defined, linear and multiplicative follows from
standard limit arguments.

\noindent b) To construct the spectral measure we define for every fixed
$x \in X$ a bounded operator $\Phi_x:C_0(\R) \to X$ by 
$\Phi_x(f) = f(A)x$ with $\|\Phi_x\| \leq 2\pi M^2\|x\|.$
Since $X$ does not contain $c_0$, it follows from \cite{DU}, Theorem 5 and 15 in
IV.2, that there is a countable additive vector measure $P_x : {\cal B} \to X$
such that $\operatorname{Var}(P_x) \leq \| \phi_x \|$ and 
(in the sense of \cite{DU}, $\S$ 1.4) for every $f \in C_0(i\R)$
$$
\phi_A(f) = \Phi_x(f) = \int_{i\R} f(\lambda)dP_x(\lambda).
$$
For a Borel function $f \in {\cal B}_b(i\R)$ we use now this formula to define
a bounded linear operator $\phi_A(f)$ on $X$. Then 
$$
\| \phi_A f \| \leq \sup \{ \| \operatorname{Var}P_x \| : \| \lambda \| \leq 1 \}
\|f\|_{\infty} \leq 2\pi M^2 \|f\|_{\infty}.
$$
The multiplicativity of $\phi_A$ on $C_0(i\R)$ can be extended to 
${\cal B}_b(i\R)$ by the following convergence property (\cite{DU}, Sect.~I.4.1):
If $f_n \in {\cal B}_b (i\R)$ is uniformly bounded and $f_n \to f$
pointwise, then $\phi_A(f_n)x \to \phi_A(f)x$ for all $x \in X$ and $n\to\infty$.
In particular, $P(\Omega) = \phi_A(\chi_{\Omega})$ defines a vector measure
with $P(\Omega)^2 = \phi_A(\chi_{\Omega}^2) = P(\Omega)$. 

\noindent c) $\{ \phi(f): f \in C_0(i\R)\}$ is $R$--bounded and $\gamma$--bounded
by 5.10 and the claim follows now from the above convergence property, since
the closure of an $R$--($\gamma$--) bounded set in the strong operator topology
is also $R$--($\gamma$--) bounded.

The converse statement follows in the same way since 
$$
T_tx = \int\limits_{i\R} e^{t \lambda} dP(\lambda)x, \quad t \in \R
$$
is the group generated by $A$.
\end{proof}

Besides the square functions $\|R(a+i \cdot,A)\|_{\gamma(\R,X)}$ one can 
consider square functions
$\|\psi(\cdot +A)\|_{\gamma(\R,X)}$ where $\psi \in H_0^{\infty}(S(a))$.

The following lemma shows, that they can be used to characterize the
$H^{\infty}$--calculus too.

\begin{lemma}
Let $A$ be an operator of strip--type on a Banach space $X$. Let $\psi$ and 
$\varphi$ be in $H_0^{\infty}(S(a))$.
Then there is a constant $C$ (depending on $A,\varphi$ and $\psi$) so that
for all $f \in H^{\infty}(S(b))$ with $b > a$ and $x \in X$
$$
\|f(A)\psi(i\cdot +A)x\|_{\gamma(\R,X)} \leq C \|f\|_{H^{\infty}(S(b))} \| 
\varphi(i\cdot +A)x\|_{\gamma(\R,X)}.
$$
In particular for $f(\lambda) \equiv 1$ we obtain
$$
\frac{1}{C} \| \psi(i\cdot +A)x\|_{\gamma(\R,X)} \leq \| \varphi(i\cdot +A)x\|_{\gamma(\R,X)} 
\leq C \| \psi(i\cdot +A)x\|_{\gamma(\R,X)}.
$$
\end{lemma}

The proof of 6.10 is similar to the proof of Proposition 7.7 below and we
omit it.

\section{$H^{\infty}$--calculus for sectorial operators}

We will consider sectorial operators with an additional $R$--bounded
assumption. A sectorial operator on a Banach space $X$ is called
{\sl $R$--sectorial} ({\sl $\gamma$--sectorial}) if for some $\omega \in (0, \pi)$
the set $\{ \lambda R(\lambda,A):|\arg \lambda| > \omega \}$ is 
$R$--bounded ($\gamma$--bounded). $\omega_R(A)$ (resp. $\omega_\gamma(A))$
is the infimum over all such $\omega$. A sectorial operator $A$ is called
{\sl almost $R$--sectorial} ({\sl almost $\gamma$--sectorial}), if there is a
$\omega \in (0, \pi)$ so that $\{ \lambda AR(\lambda,A)^2 : | \arg \lambda |
> \omega \}$ is $R$--bounded ($\gamma$--bounded). $\tilde{\omega}_R(A)$ (or
$\tilde{\omega}_\gamma(R))$ is now the infimum over all these $\omega$.

$R$--sectorial operators are almost $R$--sectorial, but not conversely
(not even in an $L_p$--space, see \cite{KW2}). The same is true for $\gamma$--sectorial
operators. We remark that these properties follow from the
existence of an $H^{\infty}$--calculus, more precisely

\begin{remark}
a) If $A$ has BIP on an arbitrary Banach space, then $A$ is almost $R$--sectorial
and almost $\gamma$--sectorial with $\tilde{\omega}_\gamma(A), \tilde{\omega}_R(A) 
\leq \omega_{BIP}(A)$. If $A$ has an $H^{\infty}$--calculus then
$\tilde{\omega}_R(A) = \omega_{H^{\infty}}(A) = \omega(A^{it})$.

b) If $A$ has BIP and $X$ has the UMD--property, then $A$ is $R$--sectorial
(see \cite{CP}). The same is true if $A$ has an $H^{\infty}$--calculus and $X$ has
only property $\Delta$ (cf.~\cite{KW1}).

c) If $X$ has type $p>1$ and $A$ is (almost) $\gamma$--sectorial, then $A^0$ is (almost)
$\gamma'$--sectorial on $X^0$ (the proof is the same as for \cite{KKW} Prop 3.5 with $\gamma(\N,X)$ in place of $\Rad X$).
\end{remark}

\begin{proof}[proof of a).]
Since $\int\limits_0^{\infty} s^{z-1} \textstyle \frac{s}{(1+s)^2} ds = 
\textstyle \frac{\pi z}{\sinh (\pi z)}$ for $\operatorname{Re}z > -1$ we obtain as a special
case of the Mellin functional calculus that
$$
\lambda A (1+ \lambda A)^{-2}x = \frac{1}{2}\int\limits_{-\infty}^{\infty}
\frac{t}{\sinh (\pi t)} \lambda^{it} A^{it} x dt
$$
for $\lambda$ with $\omega_{BIP} (A) + |\arg (\lambda)| < \pi$ and $x \in X$.
For $d = \omega_{BIP}(A) + \varepsilon$ and $\theta = \pi-d- \varepsilon$
with a small $\varepsilon > 0$ we can write for $|\arg(\lambda)| < \theta$
$$
\lambda A(1+\lambda A)^{-2}x = \int\limits_{-\infty}^{\infty}
h_{\lambda}(t) N(t)xdt, \quad x \in X
$$
with $N(t) = e^{- d|t|}A^{it}$ and $h_{\lambda}(t)= \frac{t}{\sinh(\pi t)}
\lambda^{it} e^{d|t|}$. Since $h_{\lambda}(t)$ is uniformly bounded for
$|\arg(\lambda)| \leq \theta$ the set $\{ \lambda A(1+\lambda A)^{-2} :
|\arg (\lambda)| \leq \theta \}$ is $R$--bounded and $\gamma$--bounded by 5.8.
Hence $A$ is almost $R$--bounded with $\tilde{\omega}_R(A) \leq \omega_{BIP}
(A)+2 \varepsilon$ and almost $\gamma$--bounded with $\tilde{\omega}_\gamma(A) \leq \omega_{BIP}
(A)+2 \varepsilon$ for all $\varepsilon > 0$. The last claim is shown in
\cite{KW2}.
\end{proof}

Now we can state our characterization of the $H^{\infty}$--calculus in terms
of square functions. (As in 3.2 and 4.7c) we can show that these square
functions are finite.)

\begin{theorem}
Let $A$ be an almost $\gamma$--sectorial operator on a Banach space $X$.
Consider the conditions

\noindent a)
$A$ has bounded imaginary powers and for one (all) $\omega$ with $| \omega|
\in (\omega(A^{it}), \pi]$ there is a constant $C$ with 
\begin{equation*}
\begin{split}
\|e^{- \omega|\cdot|}A^{i\cdot}x\|_{\gamma(\R,X)} & \leq C\|x\|, \quad 
x \in \cal{D}(A) \cap \cal{R}(A) \\
\|e^{- \omega|\cdot|} (A^{i\cdot})'x' \|_{\gamma'(\R,X')} & \leq C \|x'\|, \quad 
x' \in \cal{D}(A^0) \cap \cal{R}(A^0).
\end{split}
\end{equation*}

\noindent b)
For one (all) $\omega$ with $|\omega| \in (\tilde{\omega}_\gamma(A),\pi]$ there is
a constant $C$ such that
\begin{equation*}
\begin{split}
\|A^{1/2}R(\cdot e^{i \omega},A)x\|_{\gamma(\R_{+},X)} & \leq C \|x\|,
\quad x \in \cal{D}(A) \cap \cal{R}(A) \\
\|(A')^{1/2}R(\cdot e^{i \omega},A)'x'\|_{\gamma'(\R_{+},X')} & \leq C \|x'\|, 
\quad x' \in \cal{D}(A^0) \cap \cal{R}(A^0).
\end{split}
\end{equation*}

\noindent c)
For one (all) $\omega$ with $|\omega| \in (\tilde{\omega}_\gamma(A),\pi]$ there is a
constant $C$ such that for all $x \in \cal{D}(A) \cap \cal{R}(A)$
$$
\frac{1}{C} \|x\| \leq \|A^{1/2}R(\cdot e^{i \omega},A)x\|_{\gamma(\R_{+},X)} \leq
C\|x\|.
$$

\noindent d)
$A$ has an $H_{\infty}(\Sigma(\sigma))$--calculus for one (all) $\sigma
\in (\tilde{\omega}_\gamma(A),\pi]$.

Then a) $\Longrightarrow$ b) $\Longrightarrow$ c) $\Longrightarrow$ d)
always. If
$X$ has finite cotype then d) $\Longrightarrow$ a). In this case
$\omega_{H^{\infty}}(A) = \omega(A^{it}) = \tilde{\omega}_\gamma(A)$.
\end{theorem}

\begin{remark}
1) \ 
For all sectorial operators and fixed angles $\omega$ and $\sigma >
\omega$, an inspection of the proofs show that we have a) $\Longrightarrow$ b) $\Longrightarrow$ c) $\Longrightarrow$ d). 
If $X$ has finite cotype and $\sigma < \omega$ then also d) $\Longrightarrow$
a).

We need almost $\gamma$--sectoriality only to assure that our conditions are
independent of the choice of the angle.

2) \ The assumption that $A$ is almost $\gamma$--sectorial can be dropped. Indeed by the first part of the remark any of the conditions a), b), c) for some angle implies d). Now apply the second part of remark 7.1a) to $A$ and $A^0$. Now we have that $A$ and $A^0$ are almost $\gamma$--bounded and we can switch angles.

3) \ If $X'$ has also finite cotype, e.~g.~if $X$ is an $L_p(\Omega)$--space
with $1 < p < \infty$ then we can replace in the second condition of b) the
$\gamma'$--norm by the $\gamma$--norm (see 5.2):
$$
\|(A')^{1/2}R(\cdot e^{i \omega},A)'x'\|_{\gamma(\R_{+},X')} \leq C \|x'\|.
$$
\end{remark}

As a preparation for the proof of 7.2 we state

\begin{lemma}
Let $X$ have finite cotype. Suppose that the sectorial operator $A$ has an
$H^{\infty}(\Sigma(\sigma))$--calculus. Then $A$ has bounded imaginary
powers with $\omega_{BIP}(A) \leq \sigma$ and for $\theta > \sigma$ there is
a constant $C$ so that 
\begin{equation*}
\begin{split}
\|e^{- \theta|\cdot|}A^{i\cdot}x\|_{\gamma(\R,X)} & \leq C C_{H^{\infty}} e^{\pi (\theta +
\sigma)} \frac{1}{\theta - \sigma} \|x\|, \quad x \in X \\
\|e^{- \theta|\cdot|}(A^{i\cdot})'x'\|_{\gamma'(\R,X')} & \leq C C_{H^{\infty}}
e^{\pi (\theta + \sigma)} \frac{1}{\theta - \sigma} \|x'\|, \quad x' \in X'
\end{split}
\end{equation*}
where $C$ depends only on $X$ and $C_{H^\infty}$ is the bound of the 
$H^{\infty}(S(\sigma))$--calculus.
\end{lemma}

\begin{proof}
We adopt the proof of 6.4 to $T_t = A^{it}$. In particular, we must choose 
now
$$
b_{k,n} (\lambda) = \frac{1}{\sqrt{2 \pi}} \int\limits_{- \pi}^{\pi}
e^{int} h(t) g(t+k) \lambda^{i(t+k)}dt.
$$
Since $|\lambda^{i(t+k)}| \leq e^{- \sigma(t+k)}$ we can repeat the estimates
of 6.4.
\end{proof}

\begin{proof}[Proof of 7.2] 
For a fixed angle $\omega$, we can repeat the arguments in the proof of 2.1 by replacing the
Hilbert space square functions by our generalized square functions $\gamma$ and
$\gamma'$ and using in place of (S1), (S2) and (S3) their generalizations in 5.12.

In a) $\Longrightarrow$ b) we can use the formula
\begin{equation}
\frac{\pi}{\cosh(\pi s)} e^{\theta s} A^{is}x = e^{- i \theta/2} 
\int\limits_0^{\infty} t^{is} [t^{1/2} A^{1/2}(e^{i \theta}t+A)]^{-1}x
\frac{dt}{t}
\end{equation}
(which appeared already in (1) of Section 2 with $\theta < \pi - \omega(A^{it}$))
and appeal to the Plancherel formula 4.8 and 5.3 for $\gamma(\R,X)$ and $\gamma'(\R,X')$.

In b) $\Longrightarrow$ c) we use 5.5.

In c) $\Longrightarrow$ d) we follow the lines of 2.1 c) $\Longrightarrow$ d) using the operator $K$ defined there and apply 5.5 to it.

If $X$ has finite cotype then Lemma 7.4 shows that d) $\Longrightarrow$ a).

If $A$ is a $\gamma$--sectorial operator we can check that the condition 
c) does not depend on the choice of $\omega \in (\omega_\gamma(A), \pi]$ by using
formula (2) from Section 2. For an almost $\gamma$--sectorial operator we note that
$$
A^{1/2}R(t^{-1}e^{i \theta},A) = \psi_{\theta}(tA)t^{1/2} \mbox{ with }
\psi_{\theta}(\lambda) = \frac{\lambda^{1/2}}{e^{i \theta}-\lambda}
$$
and
$$
\|A^{1/2} R(\cdot e^{i \theta},A)x\|_{\gamma(\R_{+},X)} = \|\psi_{\theta}(\cdot A)x
\|_{\gamma(\R_{+},\frac{dt}{t},X)}.
$$
If $|\theta| > \tilde{\omega}_\gamma(A)$ then $\psi_{\theta} \in H_0^{\infty}
(\Sigma(\varphi))$ for $|\theta| > \varphi > \tilde{\omega}_\gamma(A)$ and we can
appeal Proposition 7.7 below to ensure the equivalence of square functions with
different angles.
\end{proof}

With 7.2 we can explain the gap between the $H^{\infty}$--calculus and BIP.

\begin{corollary}
Let $A$ be a sectorial operator on a Banach space $X$.

\noindent a)
If $A$ has BIP and $\{e^{- \theta|t|}A^{it} : t \in  \R\}$ is $\gamma$--bounded,
then $A$ has an $H^{\infty}(\Sigma(\sigma))$--calculus for $\sigma > \theta.$

\noindent b)
Conversely, assume that $X$ has property $(\alpha)$.
If $A$ has an $H^{\infty}(\Sigma(\sigma))$--calculus then $\{ e^{- \theta|t|}
A^{it}: t \in\R\}$ is $\gamma$--bounded for all $\theta > \sigma$.
\end{corollary}

\begin{proof}
a) By Example 5.7 we have for $\varepsilon > 0$
$$
\| e^{- (\theta + \varepsilon)|\cdot|}A^{i\cdot}x\|_{\gamma(\R,X)} \leq \|e^{- \varepsilon
|\cdot|}\|_{L_2(\R)} C \|x\|
$$
where $C$ is the $\gamma$--bound of $\{ e^{- \theta|t|}A^{it} \}$. The second estimate in 7.2b) follows also from 5.7.

b) Since $X$ has finite cotype $\gamma$--boundedness and $R$--boundedness are
equivalent and we can apply \cite{KW1}, Theorem 5.3.
\end{proof}

Results on the vector--valued $H^{\infty}(\Sigma(\sigma))$--calculus,
joint functional calculii and $R$--boundedness of the calculus are
contained in \cite{KW1}.

We noticed already in the proof of 7.2, that besides \newline $\|A^{1/2}R(\cdot
e^{i \omega},A)x \|_{\gamma(\R_{+},X)}$, also more general square functions
of the form
$$
\| \psi(\cdot A)\|_{\gamma(\R_{+}, \frac{dt}{t},X)}, \quad \psi \in H_0^{\infty}
(\Sigma(\sigma)),
$$
are useful. We will show shortly that they are equivalent for an almost
$\gamma$--sectorial operator. But first

\begin{lemma}
Let $A$ be almost $\gamma$--sectorial. Then we have for all $\psi \in H_0^{\infty}
(\Sigma(\omega))$ with $\omega > \tilde{\omega}_\gamma(A)$ that $\{ \psi (tA) :
t > 0 \}$ is $\gamma$--bounded.
\end{lemma}

\begin{proof}
Let $\Psi$ be an antiderivative of $\frac{\psi(\lambda)}{\lambda}$ which
vanishes at $0$. Define $\varphi(\lambda)= \Psi(\lambda) - 
\gamma \frac{\lambda}{1 + \lambda}$, where $\gamma = \int\limits_0^{\infty}
t^{-1} \psi(t)dt$. Then $\varphi'(\lambda) = \frac{\psi(\lambda)}{\lambda}
- \gamma (1 + \lambda)^{-2}$ and one can show that $\varphi \in H_0^{\infty}
(\Sigma(\omega))$. Hence 
$$
\varphi(tA) = \frac{1}{2 \pi i} \int\limits_{\partial \Sigma(\omega)}
\varphi(\lambda) R(\lambda, tA) d\lambda
$$
and therefore
\begin{align*}
tA \varphi' (tA) &= \frac{1}{2 \pi i} \int\limits_{\partial \Sigma(\omega)} \varphi(\lambda)
[tAR(\lambda,tA)^2]d\lambda \\
&= \frac{1}{2\pi i} \int\limits_{\partial \Sigma(\omega)} \left[ \frac{\varphi(t\mu)}{\mu} \right] \left[ \mu A R(\mu,A)^2 \right] d\mu \quad \text{for} \; t\in\R_{+}.
\end{align*}
Note that by 5.8 this set is $\gamma$--bounded. Since $\psi(\lambda) = \lambda\varphi'(\lambda) + \gamma\lambda(1+\lambda)^{-2}$ it follows that the set
\[ \psi(tA) = tA\varphi'(tA) + \gamma tA(1+tA)^{-2}, \quad t\in\R_{+} \]
is $\gamma$--bounded too.
\end{proof}

The next result was proved for Hilbert spaces in \cite{M} and for $R$--sectorial
operators in $L_p$--spaces in \cite{LeM1}. The proof in \cite{LeM1} can be extended to
the general setting:

\begin{proposition}
Let $A$ be an almost $\gamma$--sectorial operator on a Banach space $X$.
Let $\psi$ and $\varphi$ be in $H_0^{\infty}(\Sigma(\sigma))$ for some
$\sigma > \tilde{\omega}_\gamma(A)$.
Then there is a constant $C$ (depending on $A, \varphi$ and $\psi$), 
so that for all $f \in H^{\infty}(\Sigma(\sigma))$ and $x \in \cal{D}(A) \cap
\cal{R}(A)$ 
\begin{equation}
\|f(A) \psi(\cdot A)x\|_{\gamma(\R_{+}, \frac{dt}{t}, X)} \leq C 
\|f\|_{H^{\infty}(\Sigma(\sigma))} \| \varphi(\cdot A)x\|_{\gamma(\R_{+}, \frac{dt}{t},
X)}.
\end{equation}
In particular, for $f(\lambda) \equiv 1$, we obtain
$$
\frac{1}{C} \| \psi(\cdot A)x\|_{\gamma(\R_{+},\frac{dt}{t},X)} \leq \| \varphi(\cdot A)x
\|_{\gamma(\R_{+},\frac{dt}{t},X)} \leq C \| \psi(\cdot A)x\|_{\gamma(\R_{+},\frac{dt}{t},X)}
$$
and
$$
\frac{1}{C} \| \psi(\cdot A)'x'\|_{\gamma'(\R_{+},\frac{dt}{t},X')} \leq \| \varphi(\cdot A)'x'
\|_{\gamma'(\R_{+},\frac{dt}{t},X')} \leq C \| \psi (\cdot A)'x'\|_{\gamma'(\R_{+},\frac{dt}{t},X')}.
$$
Furthermore, if $A$ has an $H^\infty(\Sigma(\sigma))$-calculus, then for $x\in X$
\[ \frac{1}{C}\| x \| \leq \| \varphi(\cdot A) x \|_{\gamma(\R_{+},\frac{dt}{t},X)} \leq C\| x \|. \]
\end{proposition}

\begin{proof}
First we assume in addition that $f \in H_0^{\infty}(\Sigma(\theta))$.
Choose two auxiliary functions $g,h \in H_0^{\infty}(\Sigma(\theta))$
such that
$$
\int\limits_0^{\infty} g(t) h(t) \varphi(t) \frac{dt}{t} = 1.
$$
By analytic continuation we have for all $\lambda \in \Sigma(\theta)$ that
$$
\int\limits_0^{\infty} g(t \lambda) h(t \lambda) \varphi (t \lambda) 
\frac{dt}{t} = 1
$$
and we can apply the $H_0^{\infty}$--calculus of $A$ and some $\gamma \in
(\omega_{\gamma}(A), \theta)$
\begin{equation*}
\begin{split}
f(A) & = \frac{1}{2 \pi i} \int\limits_{\partial \Sigma(\gamma)}
\biggl( \int\limits_0^{\infty} g(t \lambda) h(t \lambda) \varphi (t \lambda)
\frac{dt}{t} \biggr) f(\lambda) R(\lambda,A) d \lambda \\
& = \int\limits_0^{\infty} \biggl( \frac{1}{2 \pi i} \int\limits_{\partial
\Sigma(\gamma)} g(t \lambda) f(\lambda) h(t \lambda) \varphi (t \lambda)
R(\lambda,A) d \lambda \biggr) \frac{dt}{t} \\
& = \int\limits_0^{\infty} g(tA) h(tA) f(A) \varphi(tA) \frac{dt}{t}.
\end{split}
\end{equation*}
Furthermore, by Lemma 4.1 in \cite{KW1} we have
\begin{align*}
\psi(sA)g(tA) &= \frac{1}{2 \pi i} \int\limits_{\partial \Sigma(\gamma)}
\psi(s \lambda)g(t \lambda) R(\lambda,A) d \lambda \\
&= \frac{1}{2 \pi i}
\int\limits_{\partial \Sigma(\gamma)} \psi (s \lambda) g(t \lambda)
\lambda^{\frac{1}{2}} A^{\frac{1}{2}} R(\lambda,A) \frac{d \lambda}{\lambda}.
\end{align*}
With these identities and Fubini's theorem we obtain for $x \in X$
\begin{equation*}
\begin{split}
f(A)\psi(sA)x &= \int\limits_0^\infty [\psi(sA)g(tA)]h(tA)f(A)\varphi(tA)x \frac{dt}{t} \\
& = \frac{1}{2 \pi i} \int\limits_{\partial \Sigma(\gamma)}
\psi(s \lambda) \lambda^{\frac{1}{2}} A^{\frac{1}{2}} R(\lambda,A)
\\
& \quad \cdot \biggl( \int\limits_0^{\infty} g(t \lambda) f(A)h(tA) \varphi(tA) x
\frac{dt}{t} \biggr) \frac{d \lambda}{\lambda} \\
& = \frac{1}{2\pi i} \int\limits_{\partial\Sigma(\gamma)} \psi(s\lambda) M(\lambda) \left( \int\limits_0^\infty g(t\lambda) N(t) \varphi(tA) x \frac{dt}{t} \right) \frac{d\lambda}{\lambda} \\
& = {\cal K} [ M(\lambda) {\cal L} [ N(t) \varphi(tA) x](\lambda)](s)
\end{split}
\end{equation*}
where $M(\lambda) = \lambda^{\frac{1}{2}}A^{\frac{1}{2}} R(\lambda,A),
N(t) = f(A)h(tA)$ and 
\begin{equation*}
\begin{split}
{\cal K} \varphi(s) & = \frac{1}{2 \pi i} \int\limits_{\partial \Sigma(\gamma)}
\psi (s \lambda) \varphi(\lambda) \frac{d \lambda}{\lambda}, \quad 
s \in \R_{+} \\
{\cal L} \varphi(\lambda) & = \int\limits_0^{\infty} g(t \lambda) 
\varphi(t) \frac{dt}{t}, \quad \lambda \in \partial \Sigma(\gamma).
\end{split}
\end{equation*}
These operators can be reduced to convolution operators on the
multiplicative group $(\R_{+}, \frac{dt}{t})$ with $\int\limits_0^{\infty}
| \psi (e^{\pm i \gamma} t)| \frac{dt}{t} < \infty$ and $\int |g(t
e^{\pm i \gamma})| \frac{dt}{t} < \infty$. Hence they are bounded operators
${\cal K} : L_2(\partial\Sigma(\gamma), | \frac{d \lambda}{\lambda}|) \to
L_2(\R_{+}, \frac{dt}{t})$ and ${\cal L}: L_2(\R_{+}, \frac{dt}{t}) \to
L_2(\partial\Sigma (\gamma), | \frac{d \lambda}{\lambda} |)$.
$\{ M(\lambda) : \lambda \in \Sigma(\gamma)\}$ is $\gamma$--bounded by assumption.
To see that $\{ N(t) : t \in \R\}$ is $\gamma$--bounded, note that for $t \in \R_{+}$
again by Lemma 4.1 of \cite{KW1}
\begin{equation*}
\begin{split}
N(t) & = f(A)h(tA) = \frac{1}{2 \pi i} \int\limits_{\partial \Sigma(\gamma)}
f(\lambda) h(t \lambda) R(\lambda,A) d \lambda \\
& = \frac{1}{2 \pi i} \int\limits_{\partial \Sigma(\gamma)} f(\lambda)
h(t \lambda) \lambda^{\frac{1}{2}} A^{\frac{1}{2}} R(\lambda,A) \frac{d \lambda}
{\lambda}.
\end{split}
\end{equation*}
Since $\int\limits_{\partial \Sigma(\gamma)} |f(\lambda) h(t \lambda)|
|\frac{d \lambda}{\lambda}| \leq \|f\|_{H^{\infty}(\Sigma(\theta))} 
\int\limits_{\partial \Sigma(\gamma)} | f(\lambda) | | \frac{d \lambda}{\lambda}|$
we can appeal to 5.8.
Furthermore, by extending ${\cal K}$ and ${\cal L}$ to the $\gamma$--spaces according to 4.8, 4.9 and applying 4.11 to $M(\cdot)$ and $N(\cdot)$, we obtain
\begin{equation*}
\begin{split}
\|f(A)\psi(\cdot A)x\|_{\gamma(\R_{+},\frac{ds}{s},X)} &= \| {\cal K}[M(\cdot)]{\cal L}[N(t) \varphi(tA)x](\cdot) \|_{\gamma(\R_{+},\frac{ds}{s},X)} \\
& \leq C_1 \| {\cal L} [N(\cdot) 
\varphi(\cdot A)x]\|_{\gamma(\R_{+},\frac{dt}{t},X)} \\
& \leq C_2 \|f\|_{H^{\infty}(\Sigma(\theta))} \| \varphi(\cdot A)x
\|_{\gamma(\R_{+},\frac{dt}{t},X)}.
\end{split}
\end{equation*}
For a general $f \in H^{\infty}(\Sigma(\theta))$ we use the convergence lemma
and 4.10. The last claim is shown in the same way using the corresponding properties of $\gamma'(X')$. \\
The last inequality follows from 7.2 since a sectorial operator with an $H^\infty$-calculus has BIP and is therefore also almost $R$--bounded by Remark 7.1a).
\end{proof}

For a sectorial operator $A$ on a Banach space $X$ we define the space
$X_A$ as the completion of $\cal{D}(A) \cap \cal{R}(A)$ with respect to the norm
$$
\|x\|_{X_A} = \|A^{\frac{1}{2}} R(\cdot,A)x\|_{\gamma(\R_{-},X)}.
$$
As an operator on $X_A$ $A$ has particularly good properties.

\begin{corollary}
Let $A$ be an almost $\gamma$--sectorial operator on a Banach space $X$.
Then $A$ has an $H^{\infty}(\Sigma(\sigma))$--calculus for $\sigma > \tilde{\omega}_\gamma
(A)$ as an operator on $X_A$. $A$ has an $H^{\infty}$--calculus on $X$ if 
and only if $X$ is isomorphic to $X_A$.
\end{corollary}

\begin{proof}
We can apply estimate (2) with $\psi(\lambda) = \varphi(\lambda) =
\lambda^{1/2}(1 + \lambda)^{-1}$. 
Hence there is a constant $C$ so that for all $f \in H_0^{\infty}
(\Sigma(\sigma))$ 
$$
\|f(A)x\|_{X_A} \leq C \|x\|_{X_A} \quad \mbox{ for } x \in \cal{D}(A) \cap \cal{R}(A).
$$
It follows that $A$ has an $H^{\infty}(\Sigma(\sigma))$--calculus on $X_A$.
The second claim follows from 7.2c) $\Longleftrightarrow$ d).
\end{proof}

\section{The connection between sectorial and strip--type operators}

There is an obvious parallel between the notions, results and proofs of
Section 6 and 7 if we ``apply'' the map $\lambda \in \Sigma(\sigma) \to
i \log \lambda \in S(\sigma)$. In order to make this connection more 
explicit we use the operator $\log A$, which for a sectorial
operator $A$ one can define by the extended functional calculus of \cite{CDMY}:

Put $\varphi(\lambda) = \lambda (1+ \lambda)^{-2}$. Then $\varphi(A) =
A(I + A)^{-2}$, $\varphi(A)^{-1} = 2+A+A^{-1}$ and ${\cal R}(\varphi(A)) =
{\cal R}(A) \cap {\cal D}(A)$. Also $\varphi (\cdot) \log (\cdot) \in
H_0^{\infty}(\Sigma(\sigma))$ and for $\sigma > \omega(A)$ we can define
$\log (A) = \varphi(A)^{-1} (\varphi \cdot \log)(A)$ and ${\cal D}(\log A) =
\{x \in X: (\varphi \log)(A) \in {\cal D}(\varphi(A))\}$.
This definition of $\log A$ is equivalent to the definitions in \cite{No}, cf. \cite{No}, Lemma 1 and 3.
The following integral representations show in particular that $\frac{1}{i} \log A$ is
an operator of strip--type.

\begin{lemma}
Let $A$ be a sectorial operator.

\noindent a) For $| \mbox{Im }z| > \pi$ we have
$$
(z - \log A)^{-1}  =  \int\limits_{- \infty}^{\infty} 
\frac{-1}{\pi^2 + (z-t)^2} [e^t (e^t+A)^{-1}] dt.
$$

\noindent b) For $w(\frac{1}{i}\log A) < a < \pi$ we have
$$
t^{1/2} A^{1/2} (t+A)^{-1} = \frac{1}{2 \pi i} 
\int\limits_{|\mbox{\tiny{Im} } z| = a}
\frac{t^{1/2} e^{z/2}}{t+e^z} (z - \log (A))^{-1} dz.
$$
\end{lemma}

\begin{proof}
a) The first formula is taken from \cite{No}, Satz 7. 

\noindent b) For $\lambda \in \Sigma(a)$ we get by Cauchy's formula
$$
\lambda^{1/2} (t+\lambda)^{-1} = \frac{1}{2 \pi i} 
\int\limits_{|\mbox{\tiny{Im} }z| = a} \frac{e^{z/2}}{t+e^z} (z- \log \lambda)^{-1}
dz.
$$
Since $\|(z- \log A)^{-1}\| \leq C$ for $| \mbox{Im } z| = a$ and 
$e^{z/2}(t+e^z)^{-1}$ is integrable on $\{ z: | \mbox{Im } z| = a \}$ the
properties of the $H^{\infty}$--calculus give
$$
A^{1/2} (t+A)^{-1} = \frac{1}{2 \pi i} \int\limits_{| \mbox{\tiny{Im} } z| = a}
\frac{e^{z/2}}{t+e^z} (z - \log A)^{-1}dz.
$$
\end{proof}

\begin{lemma}
a) If $A$ is a sectorial operator on $X$, then $\frac{1}{i} \log A$ is of strip--type and
$w ( \textstyle \frac{1}{i} \log A) \leq \omega (A)$.

b) If $\psi \in H_1^{\infty}(\Sigma(\sigma))$ with $\sigma > \omega(A)$,
then $\tilde{\psi}(\lambda) = \psi (e^{i \lambda})$ is in $H_1^{\infty}
(S(\sigma))$ and $\tilde{\psi}( \textstyle \frac{1}{i} \ln A) = \psi(A)$.
\end{lemma}

\begin{proof}
a) For $|\theta| \leq \pi - \omega(A) - \varepsilon$ we have that
$e^{- i \theta}A$ is sectorial and $\log (e^{-i \theta}A) = \log(A) - i \theta$
(cf.~\cite{No}). By part a) of 8.1 we have for $z$ with $| \mbox{Im }z| > \pi$
$$
-i(iz + \theta - i \log A)^{-1} = \int\limits_{- \infty}^{\infty}
\frac{1}{\pi^2+(z-t)^2} e^{i \theta}t (e^{i \theta}t+A)^{-1}dt
$$
and for $| \mbox{Re }\mu| > \omega(A)+ \varepsilon$
$$
\| (\mu - i \log A)^{-1} \| \leq C\int\limits_{- \infty}^{\infty}
\frac{1}{\pi^2+t^2} dt \cdot \sup \{ \| \lambda R(\lambda,A)\|:
|\arg(\lambda)| \geq \omega(A) + \varepsilon \}.
$$

For the details of the last estimate, see \cite{No} or \cite{Haa1}, Lemma 3.5.1

b) Clearly $\int\limits_{- \infty}^{\infty} | \tilde{\psi} (a+is)|ds =
\int\limits_0^{\infty} | \psi (e^{ia}t)| \textstyle \frac{dt}{t}$ for
$|a| < \sigma$. If $T$ is a bounded operator with a bounded inverse and
$\sigma(T) \subset \Sigma(\sigma)$, then the usual Dunford calculus implies
that $\tilde{\psi} (\textstyle \frac{1}{i} \ln T) = \psi(T)$. For 
$\varepsilon > 0$ put $T_{\varepsilon} =(A+ \varepsilon I)(\varepsilon A +
I)^{-1}$. Then $T_{\varepsilon}$ is bounded, $\sigma(T_{\varepsilon}) \subset
\Sigma(\sigma)$ and for $\omega \in (\omega(A), \sigma)$
there is a constant $C$ by \cite{No}, Lemma 3, and 8.1 such that 
$$
\| \lambda R(\lambda, T_{\varepsilon})\| \leq C, \quad 
\| R(\mu, \frac{1}{i} \log T_{\varepsilon})\| \leq C
$$
for all $\lambda$ with $| \arg \lambda | \geq \omega$ and $\mu$ with
$| \mbox{Re } \mu| \geq \omega$.
Furthermore, $\displaystyle \lim_{\varepsilon \to 0} R(\lambda,T_{\varepsilon})
x$ \linebreak $= R(\lambda,A)x$ and $\displaystyle \lim_{\varepsilon \to 0} R(\mu,
\textstyle \frac{1}{i} \log T_{\varepsilon})x = R(\mu, \log A)x$ so that by
Lebesgue's convergence theorem
$$
\psi(A)x = \lim_{\varepsilon \to 0} \int\limits_{\partial \Sigma(\omega)}
\psi(\lambda) R(\lambda, T_{\varepsilon})xd \lambda = \lim_{\varepsilon \to 0}
\psi(T_{\varepsilon})x
$$
and similarly $\tilde{\psi}(\textstyle \frac{1}{i} \log A) x = 
\displaystyle \lim_{\varepsilon \to 0} \tilde{\psi} (\textstyle \frac{1}{i}
\log T_{\varepsilon})x$.
\end{proof}

Part b) of the following theorem was shown in \cite{Ok}.

\begin{theorem}
Let $A$ be a sectorial operator on a Banach space $X$ and put $B = 
\textstyle \frac{1}{i} \log A$. 

\noindent a) $A$ has an $H^{\infty}(\Sigma(\sigma))$--calculus for $\sigma
> \omega(A)$ if and only if $B$ has an $H^{\infty}(S(\sigma))$--calculus.

\noindent b) $A$ has BIP if and only if $B$ generates a $C_0$--group and
in this case $A^{it} = e^{-tB}$ for $t \in \R$. 

\noindent c) If $A$ is $R$--($\gamma$--)sectorial then $B$ is of $R$--($\gamma$--)strip--type.
If $B$ is of $R$--($\gamma$--)strip--type, then $A$ is almost $R$--($\gamma$--)sectorial.

\noindent d) For $a > \tilde{\omega}_R(A)$ and $x \in \cal{D}(A) \cap \cal{R}(A)$
we have 
\begin{equation*}
\begin{split}
\frac{1}{\sqrt{2\pi}} \| R(\cdot, B)x\|_{\gamma(\partial S(a),X)} & \leq 
\| A^{1/2} R(\cdot, A)x\|_{\gamma(\partial \Sigma(a),X)} \\
& \leq \sqrt{2\pi} \| R(\cdot, B)x\|_{\gamma(\partial S(a),X)}.
\end{split}
\end{equation*}

\noindent e) If $\psi \in H_1^{\infty} (\Sigma(\sigma))$ for $\sigma >
\omega(A)$ then $\tilde{\psi}(\lambda) = \psi (e^{i \lambda})$ is in
$H_1^{\infty}(S(\sigma))$ and for $x \in \cal{D}(A) \cap \cal{R}(A)$
$$
\| \tilde{\psi} (-i\cdot + B)x\|_{\gamma(\R_{+},X)} = \| \psi (\cdot A)x\|_{\gamma(\R_{+}, 
\frac{dt}{t},X)}.
$$
\end{theorem}

Of course this theorem allows to derive the results of Section 7 from
Section 6. Since the case of sectorial operators is of particular interest in
applications, we prefer the direct arguments of Section 7.

\begin{proof}
a) By Lemma 8.2b) we have $\| \psi (A)\| = \| \tilde{\psi}(B)\|$ for all
$\psi \in H_1^{\infty}(\Sigma(\sigma))$. Of course for $\varphi \in 
H_1^{\infty}(S(\sigma))$ there is a $\psi \in H_1^{\infty}(\Sigma(\sigma))$
with $\varphi = \tilde{\psi}$, namely $\psi (\mu)= \varphi(\textstyle
\frac{1}{i} \ln \mu)$.

b) Put $\psi (\lambda) = \lambda^{it}$, $\tilde{\psi}(\mu) = e^{-t \mu}$
and $\psi_n (\lambda) = \varphi_n (\lambda) \lambda^{it}$,
$\tilde{\psi}_n(\mu)=$ $\varphi_n(e^{i \mu})$ $e^{- t \mu}$, where $\varphi_n
(\lambda) = \textstyle \frac{n}{n+ \lambda} - \frac{1}{1 + n \lambda}$.
By 8.2b) we have $\psi_n(A)= \tilde{\psi}_n(B)$. Let $\psi(A)$ and
$\tilde{\psi}(B)$ be the (possibly) unbounded operators defined by the
extended calculus for $A$ and $B$, respectively. 
Since $\psi_n(\lambda) \to \psi(\lambda)$ and $\tilde{\psi}_n (\mu) \to
\tilde{\psi}(\mu)$ the convergence lemma implies that $\psi(A)x = 
\tilde{\psi}(B)x$ for $x$ in a dense subset of $X$. Hence if one of these 
operators is bounded so is the other one and $A^{it} = \psi(A)$ equals
$\tilde{\psi}(B)$, the semigroup operator generated by $-B$.
For a different proof, see \cite{Ok}.

c) The first part is shown in the same way as 8.2a) using 5.8.
For the second part use the integral representation 8.1b), 5.8 and
the fact that
$$
\int\limits_{| \mbox{\tiny{Im} }z| = a} \biggl| \frac{t^{1/2}e^{z/2}}{t+e^z} 
\biggr| d|z|= \int\limits_{\partial \Sigma(a)} 
\frac{|(\lambda/t)^{1/2}|}{|1+ \lambda/t|} \biggl| \frac{d \lambda}{\lambda} 
\biggr| =: C < \infty.
$$

d) Fix $a > \omega(A)$ and $\theta = \pi -a$. Then $e^{- i \theta}A$ is
sectorial since $\theta < \pi - \omega(A)$ and we can apply formula (1) from Section 7
with $\theta = \pi -a$ so that for $x \in \cal{R}(A) \cap \cal{D}(A)$
\begin{equation*}
\begin{split}
\frac{\pi}{\cosh (\pi s)} e^{(\pi -a)s}A^{is}x & = (-i) \int\limits_0^{\infty}
t^{is} [ (e^{-ia}t)^{1/2}A^{1/2}(e^{-ia}t-A)^{-1}x] \frac{dt}{t} \\
& = (-i) \int\limits_{- \infty}^{\infty} e^{ius} [(e^{-ia} e^u)^{1/2}
A^{1/2}R(e^{-ia}e^u,A)x]du.
\end{split}
\end{equation*}
Since the Fourier transform is an isomorphism on $\gamma(\R,X)$ and substitutions which define bounded operators on $L_2$ can also be extended to $\gamma(X)$, we get
\begin{equation}
\begin{split}
\sqrt{\frac{\pi}{2}} &\biggl| \biggl| \frac{e^{\pi \cdot}}{\cosh (\pi \cdot)} 
e^{-a\cdot}A^{i\cdot}x \biggr| \biggr|_{\gamma(\R,X)}\\ 
& =  \|(e^{-ia}e^{(\cdot)})^{1/2}A^{1/2}R(e^{-ia}e^{(\cdot)},A)x\|_{\gamma(\R,X)} \\
& =  \|A^{1/2}R(e^{-ia} \cdot, A)x\|_{\gamma(\R_{+},X)}.
\end{split}
\end{equation}
If we repeat this argument with $- \theta=a - \pi$ then
\begin{equation*}
\sqrt{\frac{\pi}{2}} \biggl| \biggl| \frac{e^{- \pi \cdot}}{\cosh (\pi \cdot)} 
e^{a\cdot}A^{i\cdot}x \biggr| \biggr|_{\gamma(\R,X)} = \|A^{1/2}R(e^{ia} \cdot , A)x\|_{\gamma(\R_{+},X)}.
\end{equation*}
On the other hand, if $T_s$ is the possibly unbounded operator $\psi_s(B)$,
$\psi_s(\lambda) = e^{- s \lambda}$, defined by the extended functional 
calculus for $B$, then with $\varphi(\lambda) = \textstyle \frac{\lambda}
{(1+ \lambda)^2}$ and $\tilde{\varphi}(\mu) = \textstyle
\frac{e^{i \mu}}{(1+ e^{i \mu})^2} \in H_0^{\infty}(S(a))$
$$
R(a-it,B) \tilde{\varphi}(B) y = \int\limits_0^{\infty} e^{is} [e^{-as} T_s
\tilde{\varphi}(B)y]ds.
$$
By 4.9 we obtain for $x \in \cal{D}(A) \cap \cal{R}(A)$ (see 4.6b), 6.2)
\begin{equation}
\sqrt{2 \pi} \|e^{-a\cdot} T_{(\cdot)} x\|_{\gamma(\R_{+},X)} = \|R(a-i\cdot,B)x\|_{\gamma(\R,X)}.
\end{equation}
In the same way we obtain with $-a$ in place of $a$
\begin{equation}
\sqrt{2 \pi} \|e^{a\cdot}T_{(\cdot)} x \|_{\gamma(\R_{-},X)} = \|R(-a-i\cdot,B)x\|_{\gamma(\R,X)}.
\end{equation}
As in part b) of the proof we see that $A^{is}x = T_sx$ for $x \in \cal{D}(A) \cap
\cal{R}(A)$. Since $e^{\pi s}(\cosh \pi s)^{-1} \leq 2$ for $s \geq 0$ and
$\textstyle \frac{e^{\pi s}}{\cosh(\pi s)} e^{-2as} \leq 2$ for $s \leq 0$
we conclude from (1), (2) and (3) that for $x \in \cal{D}(A) \cap \cal{R}(A)$
\begin{equation*}
\begin{split}
\|A^{1/2} R(e^{-ia} \cdot,A)x\|_{\gamma(\R,X)} & \leq 2 \sqrt{\frac{\pi}{2}}\left(
\|e^{-a\cdot}A^{i\cdot}x\|_{\gamma(\R_{+},X)} + \|e^{a\cdot}A^{i\cdot}x\|_{\gamma(\R_{-},X)}\right) \\
& \leq \sqrt{2\pi}\left(\|R(a-i\cdot,B)x\|_{\gamma(\R,X)} + \|R(-a-i\cdot,B)x\|_{\gamma(\R,X)}\right).
\end{split}
\end{equation*}
Also by (1), (2) and since $e^{\pi s}(\cosh s)^{-1} \geq 1$ for $s \geq 0$
\begin{equation*}
\begin{split}
\|R(a-i\cdot,B)x\|_{\gamma(\R,X)} & \leq \sqrt{2 \pi} \|e^{-a\cdot}T_{(\cdot)} x\|_{\gamma(\R_{+},X)} = \sqrt{2 \pi} \|e^{-a\cdot}A^{i\cdot}x\|_{\gamma(\R_{+},X)} \\
& \leq \sqrt{2\pi} \|A^{- 1/2} R(
e^{-ia} \cdot,A)x\|_{\gamma(\R_{+},X)} .
\end{split}
\end{equation*}
The estimates for $\|A^{1/2}R(e^{ia}\cdot,A)x\|_{\gamma(\R,X)}$ and $\|R(-a-i\cdot,B)x\|_{\gamma(\R,X)}$ are similar.

e) By Lemma 8.2 and \cite{No}, Satz 5, we have that $\psi(e^tA) = \tilde{\psi}
(\textstyle \frac{1}{i} \log (e^tA))$ \linebreak $= \tilde{\psi}(B-it)$ for all $t\in\R$. Applying the $\gamma$--norm with $x\in X$ gives
$$
\| \tilde{\psi}(B-i \cdot)x\|_{\gamma(\R,X)}  = \| \psi (e^{(\cdot)} A)x\|_{\gamma(\R,X)} 
= \|\psi (\cdot A)x\|_{\gamma(\R_{+}, \frac{dt}{t},X)}.
$$
Again, the substitution in the $\gamma$--norm is justified as an extention of the substitution in the $L_2$--norm via 4.8.
\end{proof}

\begin{remark}
Suppose that $B$ generates a $C_0$--group on a Banach space $X$ with the
UMD--property. By \cite{Mo}, Theorem 4.3, there exists a sectorial operator $A$
(the so--called analytic generator of $B$), so that $A^{is} = e^{-sB}$
for all $s \in \R$. Because of the uniqueness of the generator of a
$C_0$--group we have then by 8.3b) that $B = \textstyle \frac{1}{i} \log A$.
Hence the statements of 8.3 apply to $A$ and its analytic generator $B$. 
\end{remark}

\section{Littlewood Paley $g$--functions}
\setcounter{equation}{4}

A semigroup $T_t$ with generator $A$ defined on the scale $L_p(\Omega,\mu)$ for all $1\leq p \leq \infty$ is called a symmetric diffusion semigroup if
\begin{itemize}
\item[(1)] $\| T_t f \|_{L_p(\Omega)} \leq \| f \|_{L_p(\Omega)}$ for all $t>0, 1\leq p \leq \infty$,
\item[(2)] the generator $A$ of $T_t$ is selfadjoint on $L_2(\Omega)$,
\item[(3)] $T_t f \geq 0$ for $f \geq 0$ for all $t>0, f\in L_p(\Omega)$,
\item[(4)] $T_t 1 = 1$.
\end{itemize}

In \cite{St}, Stein extended the classical $g$--function estimates of Paley-Littlewood to such symmetric diffusion semigroups. For the $g$--functions
\begin{equation*}
G_k(f)(\cdot) = \left( \int_0^\infty \left\vert t^k \left( \frac{\partial^k}{\partial t^k} T_tf \right)(\cdot) \right\vert^2 \frac{dt}{t} \right)^{\frac{1}{2}},
\end{equation*}

which are functions on $\Omega$ for $f\in L_1(\Omega) \cap L_\infty(\Omega)$ and $K\in\N$, Stein showed that for $1<p<\infty$
\begin{equation}\label{6}
\frac{1}{C_p} \| f-Q f \|_{L_p(\Omega)} \leq \| G_k(f) \|_{L_p(\Omega)} \leq C_p \| f \|_{L_p(\Omega)}
\end{equation}

where $Q f = \lim_{n\to\infty} A(\frac{1}{n}+A)^{-1}f$. In \cite{Cow}, Cowling used transference to obtain the same result for semigroups only satisfying (1) and (2) for all $1\leq p \leq \infty$. For $\sigma < \frac{\pi}{2}$ and $\beta=k\in\N$
\begin{equation*}
G_\beta(f) = g_\beta(tA) f \quad \text{with} \quad g_\beta(\lambda) = \lambda^\beta e^{-\lambda} \in H^\infty_0(\Sigma(\sigma))
\end{equation*}

and $A$ is sectorial on the subspace $\overline{\cal{R}(A)}=\cal{R}(I-Q)$, complemented in $L_p(\Omega,\mu)$ for $1<p<\infty$ (see \cite{KuWe}, Proposition 15.2). Since $\frac{d^k}{dt^k} T_t = A^k T_t$, it is clear that estimate (\ref{6}) also follows from Proposition 7.7 applied to $g_\beta$ if $A$ has an $H^\infty$-functional calculus on $H^\infty(\Sigma(\sigma))$ with $\sigma < \frac{\pi}{2}$ and $\beta\in\N$. For a semigroup $T_t$ on $L_p(\Omega,\mu)$ for a single $p\in(1,\infty)$, which is bounded analytic and satisfies (1) and (3), this was shown in \cite{KW1}, Corollary 5.2. Therefore we obtain a further extension of Stein's result (\ref{6}):

\begin{corollary}
Let $p\in(1,\infty)$ and suppose that $(-A)$ generates an analytic semigroup on $L_p(\Omega,\mu)$ such that $T_t, t>0$, satisfies (1) and (3). Then $T_t$ satisfies the Paley-Littlewood estimate (\ref{6}).
\end{corollary}

Moreover, we can get Paley-Littlewood estimates for semigroups $T_t$ on Bochner spaces $L_p(\Omega,\mu,X)$. Recall that $\Tkr_t$ stands for the tensor extension of $T_t$ to $L_p(\Omega) \otimes X$. If $\Tkr_t$ defines a semigroup on $L_p(\Omega,X)$, we denote its generator bei $\Akr = A \otimes I$ and (\ref{6}) takes now the form
\begin{equation}\label{8}
\frac{1}{C_p} \| f-E_0 f \|_{L_p(\Omega,X)} \leq \| g_\beta(\cdot \Akr)f \|_{\gamma(\R_{+}, \frac{dt}{t},L_p(\Omega,X))} \leq C_p \| f \|_{L_p(\Omega,X)}
\end{equation}

for $f\in L_p(\Omega,X)$ if $\beta\in\N$. First we verify such estimates for the Gaussian and the Poisson semigroup on $L_p(\R^n,X)$. This will give a continuous analogue to Bourgain's Paley Littlewood decomposition of $L_p(\mathbb{T},X)$, which characterizes UMD-spaces.

\begin{theorem}
Let $\Delta$ be the generator of the Gaussian semigroup $T_t$ on $L_p(\R^n)$ for $1<p<\infty$. If $X$ is a UMD-space, then $A^\alpha = (-\Delta)^\alpha \otimes I$ has an $H^\infty(\Sigma(\sigma))$-calculus on $L_p(\R^n,X)$ for all $\sigma > 0$ and all $\alpha \in (0,1)$ and satisfies the Paley-Littlewood estimate (\ref{8}) for all $\beta > 0$. Conversely, either one of these conditions implies that $X$ is a UMD-space.
\end{theorem}

\begin{proof}
$B_i = \frac{\partial}{\partial x_i} \otimes I$ generates the translation group $(U_i(t)f)(x)=f(x+te_i)$ on $L_p(\R^n,X)$. Since $U_i$ is bounded and $X$ is a UMD-space, it follows from Corollary 2 in \cite{CP} that $-B_i^2 = -\left( \frac{\partial^2}{\partial x_i^2} \otimes I \right)$ has an $H^\infty(\Sigma(\sigma))$-calculus for all $\sigma > 0$. Since the $U_i(t)$, $i=1,\ldots,n$, commute it follows from \cite{LeM0}, Theorem 1.1, that $A=-(\Delta \otimes I)=-\sum_{i=1}^n B_i^2$ has an $H^\infty(\Sigma(\sigma))$-calculus on $L_p(\R^n,X)$ for all $\sigma>0$. The same is true for $A^\alpha$ with $\alpha \in (0,1)$. Now (\ref{8}) follows from Proposition 7.7 with $\psi=g_\beta$. \\
Conversely, the boundedness of the $H^\infty$-calculus of $A^\alpha$ for some $\alpha \in (0,1)$ implies that $A$ has bounded imaginary powers. By the main result of \cite{GD}, $X$ has to be a UMD-space. If (\ref{8}) holds for some $\alpha \in (0,1)$ and $\beta = \frac{1}{2}$, then theorem 7.2 and the second remark in 7.3 show that $A^\alpha$ has an $H^\infty$-calculus and we can repeat the last argument.
\end{proof}

Note that for $\alpha = \frac{1}{2}$ this result includes the Poisson semigroup on $L_p(\R^n,X)$. For more general diffusion semigoups, we have the following partial results.

\begin{theorem}
Let $1<p_0\leq 2 \leq p_1 < \infty$. Suppose that $(-A)$ generates a bounded analytic semigroup on $L_p(\Omega,\mu)$ for $p\in[p_0,p_1]$ such that $T_t$, $t>0$, is contractive and positive on $L_p(\Omega,\mu)$ (i.e. (1) and (3) are satisfied). Suppose further that $X=[X_0,H]_{\theta}$ is a complex interpolation space of a UMD-space $X_0$ and a Hilbert space $H$ with $\theta \in (0,1)$. \\
Then $\Akr$ has an $H^\infty(\Sigma(\sigma))$-calculus for some $\sigma < \frac{\pi}{2}$ on  $L_p(\Omega,X)$ with $p\in(p_0,p_1)$ and the semigroup generated by $(-\Akr)$ satisfies the Paley-Littlewood estimate (\ref{8}) for all $\beta>0$.
\end{theorem}

\begin{proof}
By a theorem of Fendler (\cite{Fend}) the positive contraction semigroup $T_t$ on $L_p(\Omega), p_0\leq p \leq p_1$, has a dilation to a group of positive isometries $U_t$ on a space $L_p(\tilde{\Omega})$, i.e. $JT_t = PU_t J$, where $J:L_p(\Omega) \to L_p(\tilde{\Omega})$ is a positive embedding and $P:L_p(\tilde{\Omega}) \to L_p(\Omega)$ a positive projection. \newline
By a standard extension theorem for positive operators (see e.g. \cite{KuWe}, 10.14), $T_t$, $U_t$, $P$ and $J$ can be extended to a contractive semigroup $\Tkr_t$ on $L_p(\Omega,X)$ and a group of isometries $\Ukr_t$ on $L_p(\tilde{\Omega},X)$ so that $\Jkr \Tkr_t = \Pkr \Ukr_t \Jkr$. Since $X$ is a UMD space the generator  $\Bkr$ of $\Ukr_t$ on $L_p(\tilde{\Omega})$ and therefore by the dilation relation $\Akr$ on $L_p(\Omega,X)$ have an $H^\infty(\Sigma(\nu))$-calculus for all $\nu > \frac{\pi}{2}$, cf. \cite{HP}. The same is true for $\Akr$ on $L_p(\Omega,H)$ for $p_0<p<p_1$. However, on $L_2(\Omega,H)$, $\Akr$ has an $H^\infty(\Sigma(\sigma))$-calculus for the same $\sigma < \frac{\pi}{2}$ that appears in our assumption on $A$. This is true since we can extend the bounded operators $\Psi(A)$, $\Psi \in H^\infty(\Sigma(\sigma))$, to operators $\Psikr$ on $L_2(\Omega,H)$ (see e.g. \cite{KuWe}, Lemma 11.11). By complex interpolation in the $L_q(\Omega,H)$ scale, $\Akr$ has an $H^\infty(\Sigma_\mu)$-calculus with $\mu < \frac{\pi}{2}$ on $L_p(\Omega,H)$ for all $p_0<p<p_1$ (cf. \cite{KKW}, Proposition 4.9). Since $[L_p(\Omega,X_0),L_p(\Omega,H)]_\theta = L_p(\Omega,X)$, the same proposition implies now that $\Akr$ has an $H^\infty(\Sigma_\mu)$-calculus with $\mu < \frac{\pi}{2}$ on $L_p(\Omega,X)$ for $p_0<p<p_1$. Now we can apply 7.7 to $\psi=g_\beta$.
\end{proof}

\begin{remark}
The theorem holds for all Banach lattices $X$ with the UMD property. Indeed, such $X$ are always interpolation spaces $X=[X_0,H]_\theta$ with some UMD Banach lattice $X_0$ and a Hilbert space $H$ as shown in \cite{Span}. Note also that in a Banach lattice $X$, the Paley-Littlewood estimate (\ref{8}) for $\beta\in\N$ takes the more traditional form
\begin{equation*}
\frac{1}{C_p} \| f \|_{L_p(\Omega,X)} \leq \left( \int_\Omega \left\Vert \int_0^\infty \left\vert t^\beta \left( \frac{d^\beta}{dt^\beta} T_t f \right)(\cdot) \right\vert^2 \frac{dt}{t} \right\Vert_X^p d\mu(\cdot) \right)^{\frac{1}{p}} \leq C_p \| f \|_{L_p(\Omega,X)}
\end{equation*}

(see 3.6). For the Laplace operator and the Gaussian semigroup on $L_p(\mathbb{T}^n,X)$, this was already shown in \cite{Xu}, Theorem 4.1.
\end{remark}

\begin{corollary}
Let $p\in(1,\infty)$. Suppose that $(-A)$ generates a bounded analytic semigroup on $L_p(\Omega,\mu)$ such that $T_t$, $t>0$, is contractive and positive on $L_p(\Omega)$. \\
Then for every UMD-space $X$ and $\alpha \in (0,1)$, the fractional power ${\bf (A^\alpha)}^\otimes$ has an $H^\infty(\Sigma(\sigma))$-calculus for some $\sigma < \frac{\pi}{2}$ on $L_p(\Omega,X)$ and the semigroup generated by ${\bf (A^\alpha)}^\otimes$ satisfies the Paley-Littlewood estimate (\ref{8}) for all $\beta>0$.
\end{corollary}

\begin{proof}
The first part of the argument of Theorem 9.3 shows that $\Akr$ has an $H^\infty(\Sigma(\nu))$-calculus for all $\nu > \frac{\pi}{2}$ on $L_p(\Omega,X)$. Then ${\bf (A^\alpha)}^\otimes = (\Akr)^\alpha$ has an $H^\infty(\Sigma(\sigma))$-calculus with $\sigma=\alpha\nu < \frac{\pi}{2}$ for $\nu$ close enough to $\frac{\pi}{2}$. Apply now 7.7 again.
\end{proof}

\begin{remark}
For diffusion semigroups satisfying (1) to (4), the results of 9.2, 9.3 and 9.4 were shown independently by Hyt\"{o}nen in \cite{Hyt1}. Instead of using the $H^\infty$-calculus, he extends the original argument of Stein in \cite{St} to the vector-valued case. However, our results also cover semigroups defined only on a part of the $L_p$ scale.
\end{remark}

\end{document}